\documentclass[11pt]{amsart}
\usepackage[left=1in,right=1in,top=1in,bottom=1in]{geometry}

\usepackage[all,cmtip]{xy}
\usepackage{amssymb}
\usepackage[english]{babel}
\usepackage{mathrsfs}

\newtheorem{theorem}{Theorem}[section]
\newtheorem{lemma}[theorem]{Lemma}

\newtheorem{proposition}[theorem]{Proposition}
\newtheorem{corollary}[theorem]{Corollary}

\theoremstyle{definition}
\newtheorem{definition}[theorem]{Definition}

\newtheorem{example}[theorem]{Example}

\newtheorem{setting}[theorem]{Setting}
\newtheorem{remark}[theorem]{Remark}

\newtheorem{question}[theorem]{Question}

\newtheorem{hypothesis}[theorem]{Hypothesis}

\newtheorem{claim}{Claim}
\newtheorem{case}{Case}
\newtheorem{subcase}{Subcase}
\counterwithin{subcase}{case}

\def\B{\mathcal{B}}

\def\N{\mathbb{N}}

\def\Z{\mathbb{Z}}

\def\Cyc{\mathrm{Cyc}}
\def\grp#1{\langle{#1}\rangle}
\def\ssgp{$\mathrm{SSGP}$}

\def\Z{\mathbb{Z}}

\def\P{\mathbb{P}}
\def\F{\mathbb{F}}

\def\ssgp{\mathrm{SSGP}}

\def\assgp{\mathrm{ASSGP}}
\def\nss{\mathrm{NSS}}

\def\nsns{\mathrm{NSnS}}
\def\minap{\mathrm{MinAP}}
\def\MAP{\mathrm{MAP}}

\def\tdlc{t.d.l.c}

\usepackage{color}
\newenvironment{revvy}{\color{black}}{}
\newenvironment{revdw}{\color{blue}}{}
\newenvironment{revsg}{\color{red}}{}

\def\by{\begin{revvy}}
\def\ey{\end{revvy}}
\def\bd{\begin{revdw}}
\def\ed{\end{revdw}}
\def\bg{\begin{revsg}}
\def\eg{\end{revsg}}

\def \Zet[#1]{\lceil #1 \rceil}  

\def\dash#1{\bar{#1}}

\def\wprod{*}
\def\fprod{\cdot}

\def\e#1{e}

\def\nice#1{\mathscr{#1}}

\def\lett{\mathrm{lett}}

\def\wprod{*}

\def\fns{finite neighbourhood system}

\begin{document}
\title[Topologically simple and metrizable free groups]{Topologically simple and  metrizable free groups with no non-trivial NSS quotients}

\author[V.H Ya\~nez]{V\'{\i}ctor Hugo Ya\~nez}
\address{School of Mathematical Sciences and LPMC, Nankai University, Tianjin 300071, P.R. China}
\email{vhyanez@nankai.edu.cn}
\thanks{The author acknowledges the partial support of their research by the Fundamental Research Funds for the Central Universities, the National Natural Science Foundation of China (NSFC) grant 12271263, and the China Postdoctoral Science Foundation Tianjin Joint Support Program under
Grant Number 2025T001TJ}

\begin{abstract}
A topological group $G$ is said to have \emph{no small subgroup} (resp. \emph{no small normal subgroup}) if it admits an open neighbourhood of the identity containing no non-trivial subgroup (resp. \emph{normal subgroup}) of $G$. These properties are usually denoted by $\nss$ (and respectively $\nsns$). The $\nss$ property plays an important historical role in the solution to the fifth problem of Hilbert due to Gleason, Montgomery-Zippin and Yamabe for the characterization of Lie groups. In 2019, Shakhmatov and the author proved that a free group $F$ with countably infinitely many generators admits a metric Hausdorff group topology $\mathscr{T}$ which satisfies the so-called \emph{algebraic small subgroup generating property} $(\assgp)$: for each open neighbourhood $U$ of the identity of $F$, the family of subgroups contained in $U$ algebraically generates $F$.  In particular $(F,\mathscr{T})$ admits no non-trivial continuous homomorphisms to either $\nss$ or locally compact groups, making it minimally almost periodic. In this paper, we prove that $(F, \mathscr{T})$ can be made topologically simple; namely, $(F, \mathscr{T})$ contains no closed normal subgroups other than $\{e\}$ and $F$. In particular, this implies that $F$ satisfies the  \emph{no small normal subgroup} ($\nsns$) property. %
\end{abstract}

\subjclass[2010]{Primary: 22A05; Secondary: 03E99, 06A06, 20E05, 54E35, 54H11}

\keywords{topologically simple group, free group, minimally almost periodic group, small subgroup generating property}

\maketitle

\section{Introduction}

The topological spaces appearing in this paper shall be assumed to be Hausdorff. A topological group $G$ is said to have \emph{no small normal subgroup ($\nsns$)} if $G$ contains an open neighbourhood of the identity which \by contains no normal subgroup of $G$ other than $\{e\}$. \ey We say instead that $G$ has \emph{no small subgroup ($\nss$)} whenever the open neighbourhood $U$ of $G$ can be selected to contain no subgroups of $G$ other than $\{e\}$. Naturally, every $\nss$ group is $\nsns$, and a group is $\nss$ if and only if $G$ admits an open neighbourhood $U$ of the identity which contains no element $x \in G$ such that $\grp{x} \subseteq U$ (i.e, $U$ contains no \emph{cyclic elements}).

A topological group $G$ is said to be \emph{minimally almost periodic} ($\minap$) if it admits no non-trivial continuous homomorphism to a compact group. In the opposite direction, a topological group $G$ is \emph{maximally almost periodic} ($\MAP$) if its points can be separated by its family of continuous homomorphisms to compact groups. These two families of groups were introduced in the context of Functional Analysis by von Neumann \cite{vN,vN-W} as a generalization of previous work of Bohr on the topic of point-separation via the (complex-valued) almost periodic functions of a group. Any compact group is clearly $\MAP$, and it is a consequence of the Peter-Weyl theorem that any commutative and locally compact group is $\MAP$. On the other hand, the construction of $\minap$ groups is historically known to be difficult.

Some first examples of $\minap$ groups were achieved in the context of Functional Analysis by considering real vector spaces with no continuous linear functionals to the real line (see \cite{Day}). A functorial example is the so-called Hartman-Mycielski \cite{HM} construction $\mathrm{HM}(G)$ of any given topological group $G$, which was proven to be $\minap$ by Dierolf and Warken \cite{Dierolf} in 1978. Each topological group $G$ has a natural embedding into its corresponding Hartman-Mycielski construction $\mathrm{HM}(G)$ as a closed subgroup of it. $\mathrm{HM}(G)$ is known to carry a substantial amount of meaningful topological, combinatorial and algebraic properties that made it an important object in topological group theory, mathematical logic and functional analysis. In topological dynamics, the $\minap$ groups appear when one considers the \emph{extremely amenable groups}: namely, the topological groups $G$ for which any continuous action of $G$ on a compact Hausdorff space admits a fixed point. Any extremely amenable group is minimally almost periodic. The celebrated Kechris-Pestov-Todor\v{c}evi\'{c} \cite{KPT} correspondence shows a beautiful framework to fully describe the extreme amenability for certain automorphism groups of countable structures. This result has been a cornerstone for modern theory of topological dynamics and provides an immense selection of $\minap$ groups. In this theory a major open problem due to Glasner and Pestov is the following: does there exist an example of a \emph{commutative} topological group $G$ which is $\minap$ but not extremely amenable?  (see \cite[Section 3.5]{P2}). A more detailed history of this problem and of extreme amenability in the context of topological dynamics is given in a paper of Pestov \cite{P}.

A topological group $G$ is said to be \emph{topologically simple} whenever the only closed and normal subgroups of $G$ are the trivial group and $G$ itself. A topologically simple group with respect to the \emph{discrete topology} is called a \emph{simple group}. Topologically simple groups have historically been powerful tools for producing a variety of counter-examples in topological group theory. The majority of modern examples of topologically simple groups usually belong to either one of the following two categories: the locally compact and totally disconnected (\tdlc) groups, and the $\minap$ (alternatively, extremely amenable) groups.     

For simplicial trees, a combinatorial construction of simple groups due to Tits \cite{Tits} has been profoundly influential in the theory of topologically simple {\tdlc} groups; an English summary of the proof, and a short account of its impact is recorded in a manuscript of Boudec \cite{Boudec}. In more recent literature, the work of Banks, Elder and Willis \cite{Willis1} provides a generalization of Tits' combinatorial algorithms. In the general setting, a large effort in the construction of topologically simple t.d.l.c groups is due to Willis \cite{Willis2}.

Local compactness is quite often incompatible with minimal almost periodicity: it is a consequence of the Peter-Weyl theorem that an \emph{Abelian} topological group is $\minap$ if and only if it admits no non-trivial continuous homomorphisms to locally compact groups. In the general case, extreme amenability is directly incompatible with local compactness due to a result of Veech \cite{Veech}: each locally compact group admits a free action on a compact space. As a consequence, the only continuous homomorphisms from \by extremely amenable groups \ey to locally compact groups are the trivial homomorphisms. Even under the contrasting condition of the $\minap$ property, topological simplicity is still obtainable in this realm:

\begin{example}
Let $\Z$ be the integers. If $\mathcal{T}$ is a $\minap$ group topology for $\Z$, then $(\Z,\mathcal{T})$ is topologically simple: this is due to any non-trivial subgroup of $\Z$ being of finite index.
\end{example}

In the non-Abelian realm, the prototypical example is the infinite permutation group:

\begin{example}
Let $S_\infty$ be the permutation group of the natural numbers $\N$ equipped with the point-wise convergence topology. Then, $S_\infty$ is a Polish topologically simple group which is not $\nss$. As a consequence, $S_\infty$ admits no non-trivial continuous homomorphism to an $\nss$ group (in particular, to any finite-dimensional unitary matrix group), and so $S_\infty$ is $\minap$.
\end{example}

	Given a countable distance-value set $\Delta$, we denote by $\mathbb{U}_\Delta$ the $\Delta$-metric Urysohn space \cite{Gao1}. When its isometry group is equipped with the so-called \emph{permutation group} topology (which we shall denote as $\mathrm{Aut}(\mathbb{U}_\Delta)$), it was proven to be topologically simple \cite{Gao1}. A recent result of M. Etedadialiabadi, S. Gao, F. Li and R. Li \cite{Gao2} proves that when the isometry group of $\mathbb{U}_\Delta$ is given the topology of \emph{point-wise convergence in metric} (denoted as $\mathrm{Iso}(\mathbb{U}_\Delta)$), it becomes \emph{extremely amenable}. Naturally, $\mathrm{Iso}(\mathbb{U}_\Delta)$ inherits topological simplicity by virtue of carrying a coarser topology than $\mathrm{Aut}(\mathbb{U}_\Delta)$. One of their main results in \cite{Gao2} proves there exists a huge variety of extremely amenable and topologically simple groups:

\begin{theorem}[{\cite[Theorem 1.3]{Gao2}}] \label{gao:result}
There exist $\mathfrak{c}$-many pairwise non-isomorphic metrizable, extremely amenable (and topologically simple) separable groups of the form $\mathrm{Iso}(\mathbb{U}_\Delta)$. 
\end{theorem}

The following notion was first defined explicitly in \cite{SY_FreeSSGP}:

\begin{definition} \label{assgp:def}
Let $G$ be a topological group. $G$ is said to satisfy the \emph{algebraic small subgroup generating property ($\assgp$)} whenever $G = \grp{\Cyc(U)}$ holds for each open neighbourhood $U$ of the identity of $G$.
\end{definition}

The namesake of the $\assgp$ property comes from a pre-existing \emph{small subgroup generating property} ($\ssgp$) of Gould \cite{CG, Gould,Gould2} which replaces equality by topological density in Definition \ref{assgp:def}, \by and which ultimately captures the topological essence of Prodanov's algorithm \ey \cite{Prodanov} for producing elementary $\minap$ topologies on abelian groups.  A more complete dedicated background on the so-called \emph{small subgroup generating properties}  can be found in a paper of Dikranjan and Shakhmatov \cite{DS_SSGP}, in the currently available open-access survey of the author in \cite{YVH_MathSoc}, in the author's PhD thesis \cite{YVH_Thesis} and the recent \cite{YVH_Hilbert}. We note the following: 
 
\begin{remark} \label{this:remark:dec26:2024}
If $G$ is an $\assgp$ (or $\ssgp$) topological group, then $G$ admits no non-trivial continuous homomorphism to an $\nss$ group. 
\end{remark}

As a consequence of Remark \ref{this:remark:dec26:2024}, we now have: \begin{equation}
\assgp \rightarrow \ssgp \rightarrow \minap \leftarrow \text{extremely amenable}.
\end{equation}

The following question is essentially due to Comfort and Gould:

\begin{question}[{\cite[Question 5.5]{CG}}]
Is every $\assgp$ (or $\ssgp$) topological group also extremely amenable?
\end{question}

There exist groups which admit $\minap$ group topologies but no $\assgp$ group topologies. No finite power of the integers $\Z$ admits an $\assgp$ group topology (see \cite[Theorem 3.12]{CG}). The symmetric group $S_\infty$ also does not admit an $\assgp$ group topology (see \cite[Example 5.4]{DS_SSGP}).

Before its explicit definition, the $\assgp$ property had been employed implicitly by Dierolf and Warken \cite{Dierolf} to prove that %
the Hartman-Mycielski construction $\mathrm{HM}(G)$ of a topological group $G$ %
is $\minap$. %
To the best knowledge of the author, group topologies with the original unaltered property from Definition \ref{assgp:def} had not appeared in print beyond Dierolf and Warken's original until recent efforts by Shakhmatov and the author in \cite{SY_FreeSSGP}.

The main goal of this paper is to give a proof of the following result:

\begin{theorem} \label{this:main:thm}
Let $X$ be a countably infinite set, and let $F(X)$ be the free group over $X$.
Then $F(X)$ admits a metric group topology $\mathcal{T}$ which is topologically simple and satisfies the $\assgp$.
\end{theorem}

\subsection{Organization of the paper.} 

In Section 2, we present the poset-based topologization framework for the free group. We establish that finding a suitable family of dense subsets in this poset is sufficient to realize both the $\assgp$ and topological simplicity, effectively reducing the proof of Theorem \ref{this:main:thm} to Lemma \ref{topological:simplifity:set}.

In Section \ref{sec:nghbd}, we collect the notions and lemmas from \cite{SY_FreeSSGP} that shall be used for the construction of \fns\ extensions. Of particular interest to us is Definition \ref{canonical:representation:definition}. Essentially, given some \fns\ $\nice{U} = \{U_i : i \leq n\}$, our goal is to define a certain ``minimal'' extension for it which is made out of adding a number of cyclic subgroups to the last subset $U_n$; these cyclic subgroups are each generated by a ``foreign'' word from a fixed finite set $B$. We call such an extension a \emph{$B$-enrichment of $\nice{U}$}. To ensure that a $B$-enrichment $\nice{V} = \{V_i : i \leq n\}$ is itself \by a {\fns}, \ey we are \emph{required} to close the family under conjugation from the expanded alphabet $Y$ which is used to write all words from $B$. This ``closing off'' procedure produces certain ``canonical'' representations for any word appearing in $V_i$ as a product of words appearing in $V_{i+1}$, along with conjugations by the letters in $Y$. This phenomenon is precisely what is captured in Definition  \ref{canonical:representation:definition}, and it becomes crucial to fix these canonical representations for our more technical calculations later in the paper.

Section \ref{calculation:section} collects the main calculation labor that shall be needed to extend our {\fns} appropriately. Meanwhile, Section \ref{extension:section} applies our calculations to perform the appropriate extension. Let us try to explain what we are trying to do in these two sections. The main idea here is the following: assume we are given arbitrary words $g,h \in F(X)$ of which $g$ is assumed non-trivial. Then, assume we are given an arbitrary {\fns} $\nice{U} = \{U_i : i \leq n \}$ on a finite sub-alphabet $X'$ of $X$. Our goal is to produce an extension $\nice{V} = \{V_i : i \leq n \}$ of $\nice{U}$ on a larger finite sub-alphabet $Y$ of $X$ containing $X'$ with the following property: there exists some word $f \in F(Y)$ such that $f \cdot g \cdot f^{-1} \cdot h \in V_n$. This property is specifically designed to ensure that the following holds in our final constructed topology $\mathcal{T}$ for $F(X)$: \emph{for every $g \in F(X) \setminus \{e\}$, the set of conjugates $\{f \cdot g \cdot f^{-1} : f \in F(X)\}$ is dense in $F(X)$}.

Finally, Section \ref{final:final:proofsection} concludes the proof of Theorem \ref{this:main:thm} by proving that the specific subsets ensuring this dense conjugacy are indeed dense in our partial order (Lemma \ref{topological:simplifity:set}).

\section{A poset-based topologization framework for the free group}

The goal of this section is to present our partial-order framework which is used to produce group topologies on a countably generated free group $F(X)$. This technique was developed by Shakhmatov and the author in \cite{SY_FreeSSGP}. Our appeal here is to present our framework in such a generality where one can induce new topological properties by finding suitable (and at most countably many) dense sets (in the poset-theoretic sense) which shall induce the desired property from a so-called \emph{finite neighbourhood system level}.

\subsection{Basic notation and terminology}

We denote by $\N$ the set of natural numbers and we let $\N^+=\N\setminus\{0\}$. 
Given a set $X$ we denote by $[X]^{< \omega}$ the family of finite subsets of $X$. 

Let $G$ be a group.
For subsets $A,B$ of 
$G$, we let 
$$
AB=\{ab:a\in A, b\in B\}
\
\text{ and }
\
A^{-1}=\{a^{-1}:a\in A\}. 
$$
We say that a subset $A$ of $G$ is \emph{symmetric} if and only if $A = A^{-1}$.
For a subset $A$ of 
$G$, we
denote by $\grp{A}$ the smallest subgroup of $G$ containing $A$.
To simplify the notation, 
we write $\grp{x}$ instead of $\grp{\{x\}}$ for $x\in G$. Given a group $G$ and an arbitrary subset $X$ of $G$ we set $\Cyc(X) := \{x \in G : \grp{x} \subseteq X \}$.

\begin{definition}
Given a set $X$ we denote by $F(X)$ the free group over $X$. In the context of $F(X)$,  the set $X$ shall be referred to as the \emph{alphabet} of $F(X)$, and its elements shall be called its \emph{letters}. The group operation $\fprod$ for $F(X)$ is considered in the fashion, making $(F(X), \fprod)$ a group.
\end{definition}

We follow the setting and terminology previously used in \cite[Section 3]{SY_FreeSSGP} to describe products and cancelations in a free group in more detail. For convenience of the reader, we now provide a summarized list of what we consider to be the most essential terms and facts that will be required during our future arguments.

\begin{definition}[{\cite[Definition 3.2]{SY_FreeSSGP}}]
Let $X$ be a set and $F(X)$ be the free group over $X$. Given two words $v,w \in F(X)$ we shall denote by $v \wprod w$ the result of the product $v \fprod w$ whenever said product is exactly the concatenation of $v$ and $w$ in irreducible forms. %
\end{definition}

\begin{definition}[{\cite[Definition 3.11]{SY_FreeSSGP}}]
Let $X$ be a 
set. For a word $w=x_1^{\varepsilon_1} \dots x_n^{\varepsilon_n}\in W(X) \setminus \{e\}$, 
$$
\lett(w)=\{x_1,\dots,x_n\}
$$
denotes the set of all letters $x_i$ appearing in $w$. We also let 
$\lett(\e{W})=\lett(\emptyset)=\emptyset$. In the literature, the set $\lett(w)$ is also referred to as the \emph{support of $w$}. 
\end{definition}

The following summarizes basic properties of the sets of letters:

\begin{lemma}[{\cite[Lemma 3.12]{SY_FreeSSGP}}]
\label{letters:and:products}
\begin{itemize}
\item[(i)]
$\lett(v \cdot w) \subseteq \lett(v) \cup \lett(w)$ for all $v,w \in F(X)$.
\item[(ii)]
$\lett(a_1 \cdot a_2 \cdots \cdot a_m)\subseteq \bigcup_{l=1}^m \lett(a_l)$ whenever $m \in \N^+$ and 
$a_1,a_2,\dots,a_m\in F(X)$.
\end{itemize}
\end{lemma}

\subsection{Finite neighbourhood systems}

\begin{definition}
\label{SPSa}
Let $X$ be a 
set.
A {\em \fns\ of $F(X)$\/} is a 
finite
sequence
$\nice{U} = \{U_i: i \leq n\}$ (where $n\in\N^+$) 
satisfying 
the following conditions:
\begin{itemize}
\item[(1$_{\nice{U}}$)] 
$U_i\subseteq F(X)$ for every $i\le n$,
\item[(2$_\nice{U}$)] $U_{i}^{-1} = U_i$ for every 
$i\leq n$,
\item[(3$_\nice{U}$)] $ \bigcup_{x \in \dash{X}} x \fprod U_{i+1} \fprod U_{i+1} \fprod x^{-1} \subseteq U_i$ for every 
$i < n$,
\item[(4$_\nice{U}$)] 
$e\in U_n$
\end{itemize}
where $\bar{X} = X \cup X^{-1} \cup \{e\}$. 
\end{definition}

In practice, these {\fns}s shall be defined in finite sub-alphabets of an infinitely generated free group $F(X)$. These sequences are designed to approximate a future, ``full neighbourhood system'' for the entire infinitely generated free group $F(X)$. To ensure that our construction remains consistent, we introduce the notion of \emph{finite neighbourhood system extension}:  

\begin{definition}
\label{SPSb}
Given two 
sets $X$ and $Y$,
we shall say that a \fns\ $\nice{V} = \{V_i:i\leq m\}$ for $F(Y)$ is
an \emph{extension\/} of a \fns\ $\nice{U} = \{U_i: i \leq n\}$
for $F(X)$
if and only if the following conditions are satisfied:
\begin{itemize}
\item[(i$_{\nice{V}}^{\nice{U}}$)] $X \subseteq Y$, so $F(X)\subseteq F(Y)$,
\item[(ii$_{\nice{V}}^{\nice{U}}$)] $n \le m$,
\item[(iii$_{\nice{V}}^{\nice{U}}$)] $V_i\cap F(X) =U_i$ for every $i\leq n$.
\end{itemize}
\end{definition}

The crucial ingredient for extensions is given in item (iii$_{\nice{V}}^{\nice{U}}$). Essentially, this item implies that when we ``restrict'' the possible alphabet of the words in the set $V_i$ to what is considered an ``native'' (i.e, smaller) alphabet $X$, the only words which can possibly appear are the \textbf{original words} from $U_i$. When producing extensions of finite neighbourhood systems, this requirement is the hardest to obtain, as item (3$_\nice{U}$ of Definition \ref{SPSa} implies that the elements of $V_i$ have to be closed under conjugation of members of the larger, ``foreign'' alphabet $Y$). A large part of the technical details in this paper are designed to ensure that when we attempt to extend neighbourhood systems, restriction (iii$_{\nice{V}}^{\nice{U}}$) will still hold in spite of adding a huge number of ``foreign'' words.

With these notions in place, the actual construction of the topology is carried out using a forcing-like argument. First, we consider the following partial order:

\begin{definition}
\label{def:P}
Let $X$ be a countably infinite set.
\begin{itemize}
\item[(a)] 
We denote by 
$\P$ the set of all triples $p=\langle\!\langle X^p,n^p,\nice{U}^p \rangle\!\rangle$ satisfying the following conditions:
\begin{itemize}
\item[(1$_p$)] $X^p\in[X]^{<\omega}$,
\item[(2$_p$)] \by $n^p\in\N^+$, \ey
\item[(3$_p$)] $\nice{U}^p = \{U_i^p:i \leq n^p\}$ is a finite neighbourhood system for $F(X^p)$.
\end{itemize}

\item[(b)] Given triples
$
p=\langle\!\langle X^p,n^p,\nice{U}^p \rangle\!\rangle \in\P
$
and 
$q=\langle\!\langle X^q,n^q,\nice{U}^q,\rangle\!\rangle \in\P,
$ we define
$q\le p$ if and only if 
$\nice{U}^q$ is an extension of $\nice{U}^p$ in the sense of Definition \ref{SPSb}.
\end{itemize}
\end{definition}

\begin{definition}
Let $(\P,\le)$ be a poset.
Recall that a set $D\subseteq \P$ is called:
\begin{itemize}
\item[(i)] 
   {\em dense in $(\P,\le)$\/} provided that for every $p\in \P$ there exists $q\in D$ such that $q\le p$;
\item[(ii)]   
   {\em downward-closed in $(\P,\le)$\/} if $p \in D$, $q \in \P$ and $q \leq p$ imply $q \in D$.
\end{itemize}    
\end{definition}

\begin{remark}
If $\mathcal{D}_1$ and $\mathcal{D}_2$ are dense in $\P$ and $\mathcal{D}_1$ is downward-closed in $\P$, then $\mathcal{D}_1 \cap \mathcal{D}_2$ is dense in $\P$. 
\end{remark}

For a countably infinite set $X$, the pair $(\P, \le)$ is a non-empty and well-defined partial order (see \cite[Lemma 8.2 and Lemma 8.3]{SY_FreeSSGP}). 
\by In the following table we outline some subsets already proven to be dense in $\P$ by Shakhmatov and the author in  \cite[Lemma 8.6]{SY_FreeSSGP}:  \ey

\begin{lemma}[{\cite[Lemma 8.6]{SY_FreeSSGP}}] \label{dense:sets}
Given a countably infinite set $X$ and $(\P,\le)$ the partial order from Definition \ref{def:P}, the following holds:

\vspace{1em}
{\centering
\begin{tabular}{ | l  l l l l |}
Quantifier & Name & & Set & Property \\ \hline
$\forall n \in \N$ & $A_n$ & = & $\{q\in\P: n\le n^q\}$ & is dense and downward-closed in $\P$. \\
$\forall S \in [X]^{< \omega}$ & $B_S$ & = & $\{ q \in \P : S \subseteq X^q \}$ & is dense in $\P$. \\
$\forall g \in F(X)\setminus\{e\}$ & $C_g$ & = & $\{q \in \P: g\in F(X^q)\setminus U^q_{n^q}\}$ & is dense in $\P$. \\
$\forall g \in F(X)$ & $D_g$ & = & $\{q \in \P: g \in \grp{ \Cyc (U_{n^q}^q ) } \}$ & is dense in $\P$.
\end{tabular} \par
} \vspace{1em}
\end{lemma}

\subsection{Construction of a neighbourhood system}

In what follows, we always assume $X$ is a countably infinite set (the alphabet), we let $F(X)$ be the free group on $X$ and we consider the partial order $(\P,\le)$ from Definition \ref{def:P}. 

\begin{definition}
Let $\mathscr{D}_1$ and $\mathscr{D}_2$ be two non-empty families of dense subsets of $\P$. We say that $\mathscr{D}_{2}$ is \emph{a refinement of} $\mathscr{D}_1$ if the following condition holds: for each member $D_1 \in \mathscr{D}_1$, there exists a member $D_2 \in \mathscr{D}_2$ such that $D_2 \subseteq D_1$. 
\end{definition}

Given a countably infinite family $\mathscr{D}$ of dense subsets of $\P$, we may select a countably infinite linear order $\F \subseteq \P$ which intersects all members of $\mathscr{D}$ via a standard induction (in the literature, this is known as the \emph{Rasiowa-Sikorski} lemma). We note the following:

\begin{remark}
Let $\mathscr{D}_1$ and $\mathscr{D}_2$ be two countably infinite families of dense subsets of $\P$. Assume that $\mathscr{D}_{2}$ is \emph{a refinement of} $\mathscr{D}_1$. If  $\F \subseteq \P$ is a linear order which intersects all members of $\mathscr{D}_2$, then $\F$ also intersects all members of $\mathscr{D}_1$. If we say $\mathscr{D}_2$ is a refinement of $\mathscr{D}_1$, we always assume that $\mathscr{D}_2$ is a family of dense subsets of $\P$. 
\end{remark}

Let us fix the following family of dense sets of $(\P,\le)$:
\begin{equation}
\mathscr{D}_{T_2}=
\{A_n  :n \in \N\}
\cup
\{B_S: S \in [X]^{< \omega} \}
\cup
\{C_g: g\in F(X) \setminus\{e\}\}.
\end{equation}

For each countably infinite refinement $\mathscr{D}$ of $\mathscr{D}_{T_2}$, we define a sequence $\B_{\mathscr{D}} = \{U_n:n\in\N\}$ where \begin{equation} \label{this:eq:19oct2024}
U_n=\bigcup\{U_n^p:p\in\F\text{ and }n\le n^p\}.
\end{equation}

\begin{proposition} \label{this:dec18:welldefined:prop}
Let $\mathscr{D}$ be a countably infinite refinement of $\mathscr{D}_{T_2}$. Let $\F$ be a linear order of $\P$ which intersects every member of $\mathscr{D}$. Then,  for each $n \in \N$ the set $U_n$ from \eqref{this:eq:19oct2024} is well-defined, and the inclusion $e \in U_n$ holds.
\end{proposition}

\begin{proof}
Let $n \in \N$ be arbitrary. Since $\mathscr{D}$ is a refinement of $\mathscr{D}_{T_2}$ there exists some $D \in \mathscr{D}$ such that $D \subseteq A_n$. By construction of $\F$ we may select some $p \in \F \cap D$. Since $p \in A_n$, we have $n \leq n^p$, therefore $e \in  U^p_{n^p} \subseteq U^p_n \subseteq U_n$ holds by items (3$_\nice{U}$) and (4$_\nice{U}$) of Definition \ref{SPSa}, and \eqref{this:eq:19oct2024}.
\end{proof}

While the following lemma is essentially proven in \cite{SY_FreeSSGP} for a specific family $\mathscr{D}$ which refines $\mathscr{D}_{T_2}$, the framing in  \cite{SY_FreeSSGP} does not make use of the formal language of refinements that we employ here. We provide an elementary proof for the sake of completion.

\by
\begin{lemma} \label{this:hausdorff:lemma}
Let $\mathscr{D}$ be a countably infinite refinement of $\mathscr{D}_{T_2}$. Let $\F$ be a linear order of $\P$ which intersects every member of $\mathscr{D}$. The sequence $\B_{\mathscr{D}} = \{U_n:n\in\N\}$ as in \eqref{this:eq:19oct2024} forms a neighbourhood basis of the identity for some metrizable (Hausdorff) group topology $\mathcal{T}_\mathscr{D}$ for $F(X)$.
\end{lemma}

\begin{proof}
We employ the three-condition characterization given in \cite[Theorem 2.1.10 and Remark 2.1.14]{new1}. We follow the same proof strategy as in \cite[Section 9]{SY_FreeSSGP}. Let us then focus on each condition one-by-one: \begin{claim} \label{this:claim1:dec92025}
For each $n \in \N$, we have $U_{n+1} \cdot U_{n+1} \subseteq U_n$. 
\end{claim}
\begin{proof}
Let $n \in \N$ be arbitrary. Given an arbitrary pair $g_1,g_2 \in U_{n+1}$, \eqref{this:eq:19oct2024} implies there exists some $p_1,p_2 \in \F$ such that $g_i \in U^{p_i}_{n+1}$ where $n + 1 \leq n^{p_i}$ holds for $i=1,2$. Since $\F$ is linearly ordered, we may (without loss of generality) assume that $p_2 \leq p_1$. We now employ condition (3$_\nice{U}$) from Definition \ref{SPSa} to deduce that $$e \cdot g_1 \cdot g_2 \cdot e \in U^{p_2}_{n+1} \cdot U^{p_2}_{n+1} \subseteq U^{p_2}_{n}.$$
This proves that $g_1 \cdot g_2 \in U^{p_2}_{n} \subseteq U_n$. This proves the inclusion $U_{n+1} \cdot U_{n+1} \subseteq U_{n}$, and so we set $m = n+1$ per our claim.
\end{proof}
\begin{claim} \label{this:claim2:dec92025}
For each $n \in \N$ the set $U_n$ is symmetric. 
\end{claim}
\begin{proof}
Let $n \in \N$ be arbitrary. For each $p \in \F$ satisfying $n \leq n^p$, the set $U^p_n$ is symmetric due to condition (2$_\nice{U}$) from Definition \ref{SPSa}. Since $U_n$ is then a union of symmetric subsets of $F(X)$ by \eqref{this:eq:19oct2024}, we deduce that the set $U_n$ itself is symmetric. 
\end{proof}
\begin{claim} \label{this:claim3:dec92025}
For each $n \in \N$ and $g \in F(X)$ there exists $m \in \N$ such that $g \cdot U_m \cdot  g^{-1} \subseteq U_n$. 
\end{claim}
\begin{proof}
Fix $n \in \N$ and $g \in F(X)$. Fix a representation $g = a^{\varepsilon_1}_1 \cdot a^{\varepsilon_2}_2 \cdot \cdots \cdot a^{\varepsilon_l}_l$ of $g$ as an irreducible product of letters in the alphabet $X$ (where each $\varepsilon_j \in \{-1,1\}$ for $j = 1,\dots,l$). Let $Y = \{a_1,\dots,a_l\} \subseteq X$ be the set of letters used to represent $g$. We shall prove that the suitable integer $m$ is actually $m = n + l$ where $l$ is the length of $g$. To do so, we employ a short induction on $l$. Let $g \in U_{n+1}$ be arbitrary, and let $p \in \F$ satisfy $g \in U^p_{n+1}$ and $n + 1 \leq n^p$. Since $\mathscr{D}$ refines $\mathscr{D}_{T_2}$, the linear order $\F$ intersects the set $B_Y$ (from Lemma \ref{dense:sets}). Thus, we may find some $q \in \F$ such that $Y \subseteq X^q$. Using the linear order of $\F$, either $p$ extends $q$ or $q$ extends $p$. Select the one which extends the other, and call it $r$. In this situation, we have $Y \subseteq X^r$ and $n +1 \leq r$ both hold. We now employ condition (3$_\nice{U}$) from Definition \ref{SPSa} to deduce that \begin{equation*}
a^{\varepsilon_1}_1 \cdot g \cdot a^{-\varepsilon_1}_1 \in a^{\varepsilon_1}_1 \cdot U^p_{n+1} \cdot U^p_{n+1} \cdot a^{-\varepsilon_1}_1 \subseteq U^p_n \subseteq U_n.
\end{equation*}
Since $g \in U_{n+1}$ was arbitrary, we have now proven that \begin{equation*}
a^{\varepsilon_1}_1 \cdot U_{n+1} \cdot a^{-\varepsilon_1}_1 \subseteq U_n
\end{equation*}
holds. Following the same procedure, we prove that $a^{\varepsilon_2}_{2} \cdot U_{n+2} \cdot a^{- \varepsilon_2}_{2} \subseteq U_{n+1}$ thereby deducing that \begin{equation*}
(a^{\varepsilon_1}_1 \cdot a^{\varepsilon_2}_{2}) \cdot U_{n+2} \cdot (a^{- \varepsilon_2}_{2} \cdot a^{- \varepsilon_1}_1) \subseteq a_1 \cdot U_{n+1} \cdot a^{-\varepsilon_1}_1 \subseteq U_n.
\end{equation*}
Proceeding as above by induction on $l$, we prove \begin{equation*}
(a^{\varepsilon_1}_1 \cdot a^{\varepsilon_2}_2 \cdots a^{\varepsilon_l}_l) \cdot U_{n+l} \cdot (a^{- \varepsilon_l}_l \cdots \cdot a^{- \varepsilon_2}_2 \cdot a^{- \varepsilon_1}_1) = g \cdot U_{n+l} \cdot g^{-1} \subseteq U_n.
\end{equation*}
We now set $m = n+l$ to complete this claim.
\end{proof}
By Claims \ref{this:claim1:dec92025}, \ref{this:claim2:dec92025}, \ref{this:claim3:dec92025} and \cite[Theorem 2.1.10 and Remark 2.1.14]{new1} we conclude that $\B_{\mathscr{D}}$ is a basis of neighbourhoods of the identity for a group topology $\mathcal{T}_\mathscr{D}$ on $F(X)$. To prove that the topology is Hausdorff, it suffices to see the following:
\begin{claim} \label{this:claim4:dec92025}
The equality $\bigcap_{n \in \N} U_n = \{e\}$ holds. 
\end{claim}
\begin{proof}
Let $g \in F(X) \setminus \{e\}$ be arbitrary. Appealing to Lemma \ref{dense:sets}, the linear order $\F$ intersects the dense set $C_g$ non-trivially since $\mathscr{D}$ refines $\mathscr{D}_{T_2}$. Therefore, we may find some $p \in \F$ such that $g \in F(X^p) \setminus U^p_{n^p}$. Towards a contradiction, let us assume that $g \in U_{n^p}$. By \eqref{this:eq:19oct2024}, there exists some $q \in \F$ such that $n^p \leq n^q$ and $g \in U^q_{n^p}$ holds. Since $\F$ is a linear order, we may directly compare $p$ and $q$. If $p$ extends $q$, then condition (iii$_{\nice{V}}^{\nice{U}}$) from Definition \ref{SPSb} implies that $U^p_{n^p} \cap F(X^q) = U^q_{n^p}$. This implies that $g \in U^p_{n^p}$, which is a contradiction to the selection of $p$. If we now assume that $q$ extends $p$, then $U^q_{n^p} \cap F(X^p) = U^p_{n^p}$ holds. However, the selection of $p$ implied that $g \in F(X^p)$ and the selection of $q$ implied that $g \in U^q_{n^p}$. Thereby, $g \in U^p_{n^p}$ must also hold, which still contradicts the selection of $p$. Since both cases lead to a contradiction, we conclude that $g \not \in  U_{n^p}$ must hold. This proves that $\bigcap_{n \in \N} U_n \subseteq \{e\}$. The converse inclusion holds by Proposition \ref{this:dec18:welldefined:prop}.
\end{proof}
By Claim \ref{this:claim4:dec92025}, we deduce that the (unique) group topology $\mathcal{T}_\mathscr{D}$ having $\B_{\mathscr{D}}$ as a basis of neighbourhoods of the identity is Hausdorff. The classical Birkhoff--Kakutani theorem now implies that the group topology is metrizable.
\end{proof}

Let us quickly summarize the role that each of the dense subsets considered in $\mathscr{D}_{T_2}$ plays:

\begin{remark}
In Lemma \ref{this:hausdorff:lemma} each of the dense subsets is used for the following purpose:

\vspace{1em}
{\centering
\begin{tabular}{ | l  l l l l |}
Quantifier & Name & & Set & Application \\ \hline
$\forall n \in \N$ & $A_n$ & = & $\{q\in\P: n\le n^q\}$ & ensuring $\B_{\mathscr{D}}$ is countably infinite. \\
$\forall S \in [X]^{< \omega}$ & $B_S$ & = & $\{ q \in \P : S \subseteq X^q \}$ & ensuring conjugation law for $\B_{\mathscr{D}}$. \\
$\forall g \in F(X)\setminus\{e\}$ & $C_g$ & = & $\{q \in \P: g\in F(X^q)\setminus U^q_{n^q}\}$ & ensuring $\mathcal{T}_\mathscr{D}$ is Hausdorff. 
\end{tabular} \par
} \vspace{1em}
\end{remark}

\ey

\subsection{Beyond Hausdorffness: inducing the $\assgp$ and topological simplicity}

Let us now introduce a way to equip the associated group topologies $\mathcal{T}_\mathscr{D}$ we constructed previously with new properties beyond Hausdorffness. 

 We first recall the method we employed to obtain the $\assgp$ in \cite{SY_FreeSSGP}.
Consider the following family:
\begin{equation}
\label{family:D1}
\mathscr{D}_{1}=
\{C_g: g\in F(X) \setminus\{e\}\}
\cup
\{A_n \cap D_g :n \in \N, g \in F(X)\}
\cup
\{B_S: S \in [X]^{< \omega} \}.
\end{equation}

The family $\mathscr{D}_{1}$ is then countably infinite, and comprised of dense subsets of $\P$ (as proven in \cite[Section 9]{SY_FreeSSGP}). In our current terminology, we note that $\mathscr{D}_{1}$ is also a refinement of $\mathscr{D}_{T_2}$. Therefore, given any countably infinite linear order $\F$ of $\P$ which intersects all members of $\mathscr{D}_{1}$, the group topology $\mathcal{T}_{\mathscr{D}_1}$ produced via $\F$ from Lemma \ref{this:hausdorff:lemma} is Hausdorff and metrizable.

\begin{theorem}[\cite{SY_FreeSSGP}]
Let $\F$ be a linear order of $\P$ which intersects every member of $\mathscr{D}_{1}$. The group topology $\mathcal{T}_{\mathscr{D}_1}$ with the neighbourhood basis of the identity $\B_{\mathscr{D}_1} = \{U_n:n\in\N\}$ satisfies the $\assgp$.
\end{theorem}

The following lemma proves that the $\assgp$ is inherited through refinements of the family $\mathscr{D}_1$:

\begin{lemma} \label{this:lemma:jan92026}
Let $\mathscr{D}$ be a countably infinite refinement of $\mathscr{D}_1$, and let $\F$ be a linear order of $\P$ which intersects all members of $\mathscr{D}$. Then, the group topology $\mathcal{T}_{\mathscr{D}}$ with the neighbourhood basis of the identity $\B_{\mathscr{D}} = \{U_n:n\in\N\}$ produced through $\F$ satisfies the $\assgp$.
\end{lemma}

\begin{proof}
It suffices to prove that the equality $\grp{\Cyc(U_n)} = F(X)$ holds for each $n \in \N$. Let us then fix arbitrary $n \in \N$ and $g \in F(X)$. Since $\F$ intersects all members of $\mathscr{D}$, and $\mathscr{D}$ is a refinement of $\mathscr{D}_1$, then $\F$ intersects all members of $\mathscr{D}_1$. We now select some $p \in \F \cap (A_n \cap D_g)$. Two conditions hold for $p$: the inequality $n \leq n^p$ and the inclusion $g \in \grp{\Cyc(U^p_{n^p})}$. In other words, there exist elements $g_1,\dots,g_k \in U^p_{n^p}$ such that $\grp{g_i} \subseteq U^p_{n^p}$ holds for all $i = 1,\dots,k$ and $g = \prod_{i = 1}^k g_i$. Since $U^p_{n^p} \subseteq U_{n^p}$, we have that $g_1,\dots,g_k \in \Cyc(U_{n^p})$. We now use the fact that the sequence $\B_{\mathscr{D}}$ is nested (see Claim \ref{this:claim1:dec92025} from Lemma \ref{this:hausdorff:lemma}); since $n \leq n^p$ holds, the inclusion $U_{n^p} \subseteq U_n$ holds as well. This proves that $g_i \in \Cyc(U_n)$ holds for all $i = 1,\dots,k$, and so $g \in \grp{\Cyc(U_n)}$. This proves that $F(X) \subseteq \grp{\Cyc(U_n)}$, with the converse inclusion being obvious. By Definition \ref{assgp:def}, we conclude that $(F(X), \mathcal{T}_{\mathscr{D}})$ satisfies the $\assgp$. 
\end{proof}

\by \begin{definition}
Let $X$ be a non-empty set. For each word $g \in F(X)$ we define the \emph{conjugation set} of $g$ as \begin{equation}
\mathrm{Conj}(g) = \{w \fprod g \fprod w^{-1} : w \in F(X) \}.
\end{equation}
\end{definition}
\ey

For every $n \in \N$, $S \in [X]^{< \omega}$, $g \in F(X) \setminus \{e\}$ and $h \in F(X)$ we now introduce the following set:

\begin{equation} \label{this:simple:poset:dense:set}
E_{n,S,g,h} = \{q \in \P: n \leq n^q, S \subseteq X^q \text{ and } (\mathrm{Conj}(g) \cdot h) \cap U_{n^q}^q \neq \emptyset \}.
\end{equation}

Let us now consider the following family of dense subsets of $\P$:
\begin{equation} \label{family:D2}
\mathscr{D}_{2}= 
\mathscr{D}_{T_2} \cup 
\{E_{n,S,g,h}: n \in \N, S \in [X]^{< \omega}, g\in F(X) \setminus \{e\}, h \in F(X)\}.
\end{equation}

Note that $\mathscr{D}_{2}$ is countably infinite, and a refinement of $\mathscr{D}_{T_2}$. 

\by
The main technical work done in this paper is for the proof of the following lemma:

\begin{lemma} \label{topological:simplifity:set}
For every $n \in \N$, $S \in [X]^{< \omega}$, $g \in F(X) \setminus \{e\}$ and $h \in F(X)$ the set \begin{equation} %\label{this:simple:poset:dense:set}
E_{n,S,g,h} = \{q \in \P: n \leq n^q, S \subseteq X^q \text{ and } (\mathrm{Conj}(g) \cdot h) \cap U_{n^q}^q \neq \emptyset \}.
\end{equation}
is dense in $(\P,\le)$.
\end{lemma}

The proof is given at the final section \ref{final:final:proofsection}

\ey

\begin{theorem} \label{this:thm:topsimp}
Let $\mathscr{D}$ be a countably infinite refinement of $\mathscr{D}_2$, and let $\F$ be a linear order of $\P$ which intersects all members of $\mathscr{D}$. Then, the group topology $\mathcal{T}_{\mathscr{D}}$ with the neighbourhood basis of the identity $\B_{\mathscr{D}} = \{U_n:n\in\N\}$ produced through $\F$ satisfies the following property: for each $g \in F(X) \setminus \{e\}$ the set $\mathrm{Conj}(g)$ is topologically dense. 
\end{theorem}

\begin{proof}
It suffices to prove the following: for any \by $n \in \N$ \ey and $h \in F(X)$ the intersection $\mathrm{Conj}(g) \cap (U_n \fprod h)$ is non-empty.  Let us consider a finite set $Z \subseteq X$ such that $g, h \in F(Z)$. Consider the corresponding set $E_{n,Z,g,h^{-1}}$ as in \eqref{this:simple:poset:dense:set}. Since $\mathscr{D}$ is a refinement of $\mathscr{D}_2$, the set $\F$ intersects all members of $\mathscr{D}_2$. Therefore, there exists $p \in \F \cap E_{n,Z,g,h^{-1}}$. By the definition of $E_{n,Z,g,h^{-1}}$, we have \begin{equation*}
(\mathrm{Conj}(g) \fprod h^{-1}) \cap U^p_{n^p} \neq \emptyset.
\end{equation*}
Given $w \in (\mathrm{Conj}(g) \fprod h^{-1}) \cap U^p_{n^p}$ we have that $w = u \fprod h^{-1} \in U^p_{n^p}$ with $u \in \mathrm{Conj}(g)$. Therefore $w \cdot h \in (U^p_{n^p} \fprod h) \subseteq (U_{n^p} \fprod h)$ by \eqref{this:eq:19oct2024}. This proves that \begin{equation} \label{this:eq:123:oct312024}
\mathrm{Conj}(g) \cap (U_{n^p} \fprod h) \neq \emptyset.
\end{equation}
We now employ the fact that $\B_{\mathscr{D}}$ is nested (see Claim \ref{this:claim1:dec92025} from Lemma \ref{this:hausdorff:lemma}). The inequality $n \leq n^p$ implies $U_{n^p} \subseteq U_n$. Therefore, \eqref{this:eq:123:oct312024} implies that \begin{equation*}
\mathrm{Conj}(g) \cap (U_{n} \fprod h) \neq \emptyset.
\end{equation*}
Since $n \in \N$ and $h \in F(X)$ were arbitrary, we conclude that $\mathrm{Conj}(g)$ is topologically dense in $(F(X), \mathscr{T})$.
\end{proof}

The following observation is immediate:

\begin{remark}
Let $G$ be a non-trivial topological group. If the set  $\mathrm{Conj}(g)$ is topologically dense in $G$ for each $g \neq e_G$, then $G$ is topologically simple. 
\end{remark}

\begin{corollary} \label{this:coro:jan92026}
Let $X$ be a countably infinite set. The free group $F(X)$ over $X$ admits a group topology $\mathscr{T}$ which is topologically simple and satisfies the $\assgp$. 
\end{corollary}

\begin{proof}
Consider the family \begin{equation} \label{family:D3}
\mathscr{D}_{3}= \mathscr{D}_{1} \cup
\mathscr{D}_{2} \cup 
\{E_{n,S,g,h}: n \in \N, S \in [X]^{< \omega}, g\in F(X) \setminus \{e\}, h \in F(X)\}.
\end{equation}

Now $\mathscr{D}_3$ is countably infinite, and refines both $\mathscr{D}_2$ and $\mathscr{D}_1$. Select a linear order $\F$ of $\P$ which intersects all members of $\mathscr{D}_3$ and let $\mathscr{T}$ be the group topology with the neighbourhood basis of the identity $\B_{\mathscr{D}_3} = \{U_n:n\in\N\}$ produced through $\F$. Now $(F(X),\mathscr{T})$ is topologically simple and satisfies the $\assgp$ by Lemma \ref{this:lemma:jan92026} and Theorem \ref{this:thm:topsimp}.
\end{proof}

Since every topologically simple group is $\nsns$, the topologies constructed in Corollary \ref{this:coro:jan92026} provide a very strong contrast between the $\nss$ and $\nsns$ properties:

\begin{corollary}
Let $(F(X),\mathcal{T})$ be as in the conclusion of Theorem \ref{this:main:thm}. Then, there exists an open neighbourhood $U$ of the identity of $G$ with the following properties: \begin{itemize}
\item[(i)] $F(X)$ is algebraically generated by the union of all the subgroups of $G$ contained in $U$, and
\item[(ii)] $U$ contains no non-trivial normal subgroup of $F(X)$.
\end{itemize}
\end{corollary}

Recall that $S_\infty$ is topologically simple and admits no non-trivial continuous homomorphisms to $\nss$ groups. However, the group $S_\infty$ can never be equipped with an $\assgp$ group topology (see \cite[Example 5.4]{DS_SSGP}). Our result shows that a free group $F(X)$ of countable rank with our topology is rather similar to $S_\infty$ with respect to their behavior of their continuous homomorphic images, with a clear dividing line being Polishness (of $S_\infty$) vs the $\assgp$ (of $F(X)$).

An additional observation is the following:

\begin{corollary}
Let $\mathrm{adTopSimp}$ be the class of groups which admit a topologically simple group topology. Then, every countable group $G$ is realized as a quotient of a group $F(G) \in \mathrm{adTopSimp}$.
\end{corollary}

\section{Finite neighbourhood systems and their canonical representations}
\label{sec:nghbd}

Proofs of the cited statements in this section are given explicitly in \cite{SY_FreeSSGP}.

\begin{definition}[{\by\cite[Definition 4.4]{SY_FreeSSGP}\ey}]
\label{def:enrichment}
Let $\nice{U} = \{U_i : i \leq n\}$ be a finite sequence of subsets of $F(X)$ for some set $X$.
\begin{itemize}
\item[(i)]
Let $B\subseteq F(X)$. Define 
\begin{equation}
\label{eq:def:Vn}
V_n=U_n\cup B.
\end{equation}
By finite reverse induction on $i=n-1,n-2,\dots,0$, define
\begin{equation}
\label{eq:def:Vi}
V_i=U_i\cup\bigcup_{x\in \dash{X}} x\fprod V_{i+1}\fprod V_{i+1}\fprod x^{-1}.
\end{equation}
We shall call the sequence $\nice{V}=\{V_i:i\le n\}$ the {\em $B$-enrichment of the sequence $\nice{U}$ in $F(X)$\/}.
\item[(ii)]
For a set $C\subseteq F(X)$, we shall call the $\left(\bigcup_{c\in C} \grp{c}\right)$-enrichment of 
$\nice{U}$ in $F(X)$ the {\em cyclic $C$-enrichment of $\nice{U}$ in $F(X)$\/}.
\end{itemize}
\end{definition}

\by
In a future technical lemma, we shall need to consider enrichments done with the \emph{neutral element} $e$ with respect to a larger alphabet. This type of enrichment introduces conjugations with respect to the new letters in the larger alphabet:

\begin{remark}
Let $n \geq 1$, and $X$ and $Y$ be non-empty sets such that $X \subseteq Y$ and $Y \setminus X \neq \emptyset$. Assume that $\nice{U} = \{U_i : i \leq n\}$ is a finite sequence of non-empty subsets of $F(X)$. Let $\nice{V}=\{V_i:i\le n\}$ be the cyclic $\{e\}$-enrichment of $\nice{U}$ in $F(Y)$. If there exists an element $g \in U_n \setminus \{e\}$, then $\nice{V} \neq \nice{U}$. Indeed, for each $y \in Y \setminus X$ one has $y \cdot g \cdot y^{-1} \in V_{n-1}$ (since $e \in V_n$). However, since $g \in F(X) \setminus \{e\}$ and $y \not \in X$, the product $y \cdot g \cdot y^{-1}$ is already in irreducible form and prominently features the foreign letter $y$. Because $U_j \subseteq F(X)$ by definition, it follows that $y \cdot g \cdot y^{-1} \not \in U_j$ for any $j \leq n$.
\end{remark}
\ey

\begin{lemma}[{\cite[Lemma 4.5]{SY_FreeSSGP}}]
\label{enrichment:lemma}
Let $X$ be a set and $\nice{U}=\{U_i:i\le n\}$ be a finite sequence
such that:
\begin{itemize}
\item[(a)] $U_i\subseteq F(X)$ for all $i\le n$,
\item[(b)] $U_i^{-1}=U_i$ for every $i\le n$,
\item[(c)] $e\in U_n$. 
\end{itemize}
Furthermore, 
let $B\subseteq F(X)$ be a set satisfying
\begin{itemize}
\item[(d)] $B^{-1}=B$.
\end{itemize}
Then the $B$-enrichment of $\nice{U}$ in $F(X)$ is a \fns\ for $F(X)$.
\end{lemma}

\begin{definition} [{\cite[Definition 4.8]{SY_FreeSSGP}}]
\label{extension:def}
Let $X$ and $Y$ be sets such that $X\subseteq Y$. For a \fns\ $\nice{U}$ for $F(X)$, we shall denote by $\nice{U}_Y$ the 
cyclic $(Y\setminus X)$-enrichment
of $\nice{U}$ in $F(Y)$.
\end{definition}

\begin{lemma}[{\cite[Lemma 5.3]{SY_FreeSSGP}}]
\label{cyclic:C-enrichment}
Let $X$ be a 
set and $\nice{U}$ be a
\fns\  for $F(X)$.
Then for every set $Y$ containing $X$ and each 
set $C\subseteq \grp{Y\setminus X}$,
the cyclic $C$-enrichment $\nice{V}$
of $\nice{U}$ extends it.
\end{lemma}

\begin{definition}[{\cite[Definition 6.1]{SY_FreeSSGP}}]
\label{canonical:representation:definition}
Assume that $\nice{U} = \{U_i : i \leq n\}$ is a finite sequence of subsets of $F(X)$, $B\subseteq F(X)$ and $\nice{V}=\{V_i:i\le n\}$ is the $B$-enrichment of $\nice{U}$ in $F(X)$. 
By
finite reverse induction 
on $i=n,n-1,\dots,0$, we shall define a (not necessarily unique)
{\em canonical representation\/}
\begin{equation}
\label{eq:h:representation}
h=a_1\fprod a_2\fprod \cdots\fprod a_m
\
\text{ (for a suitable }
m\in\N^+)
\end{equation}
of every element $h\in V_i$ 
as follows.

{\em Basis of induction\/}.
For $h\in V_n$, we let $h=a_1$ be a canonical representation of $h$.

{\em Inductive step\/}.
Suppose that $i$ is an integer satisfying $0\le i<n$ and we have already defined a canonical representation of every element
$h\in V_{i+1}$.
We fix $h\in V_i$ and define its canonical representation as in 
\eqref{eq:h:representation} according to the rules outlined below.
By \eqref{eq:def:Vi}, at least one (perhaps both) of the following cases hold.

{\sl Case 1\/}. $h\in U_i$. In this case, we let $h=a_1$ be a canonical representation of $h$.

{\sl Case 2\/}. 
$h=x\fprod u\fprod v\fprod x^{-1}$ for 
suitable
  $x\in\dash{X}$ and 
$u,v\in V_{i+1}$. 
Suppose also that 
\begin{equation*}
u=b_1\fprod b_2\fprod\cdots\fprod b_{m_1}
\
\text{ and }
\ 
v=c_1\fprod c_2\fprod\cdots\fprod c_{m_2},
\end{equation*}
are some canonical representations of $u$ and $v$, respectively.
(These canonical representations were already defined, as $u,v\in V_{i+1}$.)
Then we call 
\begin{equation*}
\label{eq:h:gappei}
h=x\fprod 
b_1\fprod b_2\fprod\cdots\fprod b_{m_1}
\fprod
c_1\fprod c_2\fprod\cdots\fprod c_{m_2}
\fprod x^{-1}
\end{equation*}
a canonical representation of $h$.
\end{definition}

\begin{lemma}[{\cite[Lemma 6.2]{SY_FreeSSGP}}]
\label{reduction:lemma}
Let $X$ be a set and $\nice{U}=\{U_i:i\le n\}$ 
be a finite sequence satisfying items (a), (b) and (c) of Lemma
\ref{enrichment:lemma}.
Assume that $B'\subseteq B\subseteq F(X)$ and 
$B'$ is symmetric.
Let
$\nice{V}'=\{V_i':i\le n\}$
and
$\nice{V}=\{V_i:i\le n\}$ 
be the $B'$-enrichment and the $B$-enrichment of $\nice{U}$
in $F(X)$, respectively.
Suppose that 
\begin{equation*}
\label{eq:12:g}
(B\setminus B')\cap \left(\bar{X}\cup\bigcup_{i=1}^n V_i'\right)\subseteq \{e\}.
\end{equation*}

Let $\eta:F(X)\to F(X)$ be the map defined by
\begin{equation*}
\label{eq:a_l':new}
\eta(g)=
\begin{cases}
e &\text{if } g\in B\setminus B'\\
g & \text{otherwise}.
 \end{cases}
\end{equation*}

If $i\le n$ and $h=a_1\fprod a_2\fprod\cdots\fprod a_m$ is a canonical representation  
of some element $h\in V_i$,
then
$
h'=\eta(a_1)\fprod \eta(a_2)\fprod\cdots\fprod \eta(a_m)$ is a canonical representation of some element $h'\in V_i'$.
\end{lemma}

\begin{lemma}[{\cite[Lemma 6.3]{SY_FreeSSGP}}]
\label{extension:set:c}
Let $X$ and $Y$ be sets such that $X\not=\emptyset$ and $X\subseteq Y$.  Let $\nice{U}=\{U_i:i\le n\}$ be a \fns\ for $F(X)$ and let $\nice{V}=\{V_i:i\le n\}$ be its 
cyclic $(Y\setminus X)$-enrichment
in $F(Y)$. Then
\begin{equation*}
\label{sum:of:letters}
\sum_{l=1}^m |\lett(a_l)|\le |X|\cdot 4^{n-i}
\end{equation*}
whenever $i \leq n$
and 
$h=a_1\fprod a_2\fprod\cdots\fprod a_m$ is a
canonical representation of $h\in V_i$ as in Definition 
\ref{canonical:representation:definition}.
\end{lemma}

\section{Cancelation of special foreign words}
\label{calculation:section}

\by 

Let $X$ be a set, and let $g$ and $h$ be arbitrary words in $F(X)$ where $g$ is assumed non-trivial. We let $k \in \N$ be a positive integer. Given the four parameters $(X,g,h,k)$ construct the following setting in this section:

\begin{setting}[$X,Y,g,h,k$] \label{this:master:set}
We let $Y$ be a non-empty set containing $X$ such that $Y \setminus X = \{y_1,\dots,y_k\}$. We let $f \in F(Y)$ be the word defined as $f = \prod_{i = 1}^k y_i$. We then set \begin{equation} \label{this:special:word:feb5:2026}
g_0 = f \wprod g \wprod f^{-1} \wprod h.
\end{equation}

\end{setting}

Under the above setting, we introduce the following master hypothesis:

\begin{hypothesis} \label{this:master:hypothesis}
Under Setting \ref{this:master:set} $(X,Y,g,h,k)$, assume we are given a word $h^* \in F(Y)$ of the form \begin{equation} \label{this:eq101:oct292024}
h^* = w_1\fprod v_1\fprod w_2\fprod v_2\fprod\cdots\fprod w_n\fprod v_n\fprod w_{n+1}
\end{equation}
where $\{w_i : i \leq n+1\} \subseteq F(Y)$, which satisfies the following conditions: \begin{itemize}
\by
\item[(i)] There exists a letter $y_{j_0} \in Y \setminus X$ such that $y_{j_0} \not \in \lett(w_i)$ for any $i \in \{1, \dots, n+1\}$,
\ey
\item[(ii)] $v_i \in \{g_0,g_0^{-1}\}$ for each $i \in \{1, \dots, n\}$, and
\item[(iii)] $h^* \in F(X)$. 
\end{itemize}
\end{hypothesis}

Our \textbf{final goal} in this section is to prove the following statement: the word $h^*$ given under Hypothesis \ref{this:master:hypothesis}, is equal to the word $w^* = w_1 \cdot e \cdot w_2 \cdot \cdots w_n \cdot e \cdot w_{n+1}$ which is obtained from $h^*$ by all the words $v_1,\dots, v_n$ replaced by $e$. The proof of this statement \emph{may} already be intuitively true for the reader (because the words $v_i$ simply contain a \emph{some} foreign letter from those appearing possibly in the $w_i$). However, in practice, we shall need to perform several formal computations in order to prove that the statement holds. 

The proof is decomposed in several lemmas. We shall start with two lemmas; these two lemmas aim to prove the following: a product $h^*$ of the form \eqref{this:eq101:oct292024} \textbf{cannot possibly} be a word in $F(X)$ if all words $v_i$ (for $i = 1, \dots, n$) are the same. Namely, whenever all instances of $g_0$ in \eqref{this:eq101:oct292024} have the same \emph{sign} (be it positive, or negative). 

The following observation \textbf{shall be used throughout} the proofs: since $\lett(f) \cap X = \emptyset$, the product $g_0 = f \cdot g \cdot f^{-1} \cdot h = f \wprod g \wprod f^{-1} \wprod h$ is in irreducible form. %

In addition, for a given index $j_0 \in \{1,\dots,k\}$ we define $f_{j_0}$ to be the sub-word of $f$ set as \begin{equation*}
f_{j_0} = \prod_{i = j_0}^k y_i.
\end{equation*}

\begin{lemma} \label{this:claim:oct292024}
Under Hypothesis \ref{this:master:hypothesis}, assume there exists some \by $1 \leq N \leq n$ \ey for which \by $v_1 = v_2 = \dots = v_N = g_0$. \ey Then there exists $a \neq e$ with the following properties: \begin{itemize}
\item[(i)] One may represent the product \begin{equation} \label{this:eq10:oct262024}
w_1 \cdot v_1 \cdot w_2 \cdot \cdots \cdot w_{N} \cdot v_N = w'_1 \wprod f' \wprod f_{j_0} \wprod a \wprod f^{-1} \wprod h
\end{equation}
in irreducible form where $w'_1$ and $f'$ are irreducible initial sub-words of $w_1$, and $f$ respectively. 
\item[(ii)] $(g^l \cdot a \cdot g^k) \neq e$ holds for any pair $l,k > 1$; and
\item[(iii)] for any pair $l,k \geq 0$ the word $(g^l \cdot a \cdot g^k)$ starts and ends with the exact same letters as $g$ does. 
\end{itemize}
\end{lemma}

\begin{proof}
We prove this by induction on $k$; namely, the length of the sequence in \eqref{this:eq10:oct262024}. If $N=1$ the sequence is of the form $w_1 \cdot v_1 = w_1 \cdot (f \cdot g \cdot f^{-1} \cdot h)$. Since $y_{j_0} \not \in \lett(w_1)$, the product $w_1 \cdot f$, when computed, is of the form \begin{equation}
w_1 \cdot f = w'_1 \wprod f' \wprod f_{j_0} = w'_1 \wprod \prod_{i = j}^{j_0-1} y_i \wprod \prod_{i = j_0}^{k} y_i.
\end{equation}
where $w'_1$ is an irreducible initial sub-word of $w_1$, $f'$ is an irreducible sub-word of $f$ where the integer $j$ satisfies $1 \leq j < j_0 \leq k$. This now implies that \begin{equation} \label{this:eq:feb3:2026}
w_1 \cdot v_1 = w'_1 \wprod f' \wprod f_{j_0} \wprod g \wprod f^{-1} \wprod h
\end{equation}
We then let $a = g$ to complete this case. Assume that the claim holds successfully for some $N \in \{1,\dots,n-1\}$. We now prove it for a sequence $w_1 \cdot v_1 \cdot w_2 \cdot \cdots \cdot v_{N} \cdot w_{N+1} \cdot v_{N+1}$ of length $N+1$. By induction hypothesis we select some $a' \neq e$ such that \begin{equation*}
w_1 \cdot v_1 \cdot w_2 \cdot \cdots \cdot w_{N} \cdot v_N = w'_1 \wprod f' \wprod f_{j_0} \wprod a' \wprod f^{-1} \wprod h
\end{equation*}
is irreducible, and $a'$ satisfies conditions (ii)--(iii) as required in the claim. We first compute $w_{N+1} \cdot f$ in analog fashion to \eqref{this:eq:feb3:2026}: since $y_{j_0} \not \in \lett(w_{N+1})$ we have \begin{equation} \label{thiseq:10:feb3:2026}
w_{N+1} \cdot f = w'_{N+1} \wprod f'' \wprod f_{j_0} = w'_{N+1} \wprod \prod_{i = j'}^{j_0-1} y_i \wprod \prod_{i = j_0}^{k} y_i.
\end{equation}
where $w'_{N+1}$ is an initial sub-word (a prefix) of $w_{N+1}$ and $f'' = \prod_{i = j'}^{j_0 - 1} y_i$ is an irreducible subword of $f$ where the integer $j'$ satisfies $1 \leq j' < j_0 \leq k$. Now \textbf{our goal is to compute the product} \begin{equation} \label{thiseq:11:feb3:2026}
(w'_1 \wprod f' \wprod f_{j_0} \wprod a' \wprod f^{-1} \wprod h) \cdot (w'_{N+1} \wprod f'' \wprod f_{j_0} \wprod g \wprod f^{-1} \wprod h),
\end{equation}
which is exactly $w_1 \cdot v_1 \cdot w_2 \cdot \cdots \cdot v_{N} \cdot w_{N+1} \cdot v_{N+1}$.
The first crucial calculation is $h \cdot (w'_{N+1} \wprod f'' \wprod f_{j_0})$. Since $h \in F(X)$, the product $h \cdot w'_{N+1}$ does \textbf{not} contain the letter $y_{j_0}$ in its irreducible form. In the worst case, $h \cdot w'_{N+1}$ can at most cancel all of $f''$, but never have cancelations with $f_{j_0}$ (which starts with the letter $y_{j_0}$). Therefore the irreducible form of $h \cdot w'_{N+1}$ concatenates with a subword of $f$ as follows: there exists an integer $j''$ satisfying $j' \leq j'' < j_0$ such that \begin{equation} \label{thiseq:20:feb3:2026}
h \cdot (w'_{N+1} \wprod f'' \wprod f_{j_0}) = (h \cdot w'_{N+1}) \wprod \prod_{i = j''}^{j_0-1} y_i \wprod f_{j_0}
\end{equation} 
\begin{case}
Let us first consider the case where the cancellation between $h$ and $w'_{N+1}$ \textbf{does not fully consume $h$} (namely, the irreducible form of $h \cdot w'_{N+1}$ retains an initial sub-word of $h$). 

In this case, \eqref{thiseq:10:feb3:2026} implies that $f^{-1} \cdot (h \cdot w'_{N+1}) = f^{-1} \wprod (h \cdot w'_{N+1})$ (namely, $f^{-1}$ concatenates with the irreducible form of $h \cdot w'_{N+1}$). In this event, \eqref{thiseq:11:feb3:2026} is now computed as \begin{equation*}
w'_1 \wprod f' \wprod f_{j_0} \wprod a' \wprod f^{-1} \wprod (h \cdot w'_{N+1}) \wprod \prod_{i = j''}^{j_0-1} y_i \wprod f_{j_0} \wprod g \wprod f^{-1} \wprod h. 
\end{equation*}
We note that the above still holds even if $h \cdot w'_{N+1} = h$ (i.e, $w'_N$ is the empty word): this is because $h$ contains letters which are only in $X$. Here we now define $a := a' \wprod f^{-1} \wprod (h \cdot w'_{N+1}) \wprod \prod_{i = j''}^{j_0-1} y_i \wprod f_{j_0} \wprod g$ and note that it starts and ends with the exact same letters as $g$ (as it is constructed as direct concatenation, and by our hypothesis on $a'$), ensuring (ii) and (iii) are satisfied. 
%\end{proof}
\end{case} 
In what follows, we assume the opposing case: the cancellation fully consumes $h$. Namely, $h \cdot w'_{N+1}$ is simply a final subword of $w'_{N+1}$ (\textbf{perhaps empty}). \textbf{Let us denote that sub-word as $w''_{N+1} = h \cdot w'_{N+1}$}.

\begin{case} Let us first consider the possibility where $w''_{N+1} = h \cdot w'_{N+1}$ is trivial. In this event, we are now concerned with the product \begin{equation*}
f^{-1} \cdot (h \cdot w'_{N+1}) \cdot \prod_{i = j''}^{j_0-1} y_i \wprod f_{j_0} = f^{-1} \cdot (\prod_{i = j''}^{j_0-1} y_i \wprod f_{j_0})
\end{equation*}
\begin{subcase} If $j'' > 1$, then the above product is simply a concatenation of the form
\begin{equation*}
f^{-1} \cdot (\prod_{i = j''}^{j_0-1} y_i \wprod f_{j_0}) = f^{-1} \wprod (\prod_{i = j''}^{j_0-1} y_i \wprod f_{j_0})
\end{equation*}
given that the letters $y_1$ and $y_{j''}$ would be distinct). We now are able to compute \eqref{thiseq:11:feb3:2026} as \begin{equation*}
w'_1 \wprod f' \wprod f_{j_0} \wprod a' \wprod f^{-1} \wprod \prod_{i = j''}^{j_0-1} y_i \wprod f_{j_0} \wprod g \wprod f^{-1} \wprod h.
\end{equation*}
We now define $a := a' \wprod f^{-1} \wprod \prod_{i = j''}^{j_0-1} y_i \wprod f_{j_0} \wprod g$ to complete this case, wherein $j'' > 1$. 
\end{subcase}

\begin{subcase} If $j'' = 1$ (and recall, $h \cdot w'_{N+1}$ is trivial), then we have further cancelations in \eqref{thiseq:11:feb3:2026}: \begin{equation} \label{thiseq:12:feb3:2026}
(w'_1 \wprod f' \wprod f_{j_0} \wprod a' \wprod f^{-1}) \cdot (h \cdot w'_{N+1}) \cdot (\prod_{i = j''}^{j_0-1} y_i \wprod f_{j_0} \wprod g \wprod f^{-1} \wprod h) = (w'_1 \wprod f' \wprod f_{j_0} \wprod a') \cdot (g \wprod f^{-1} \wprod h)
\end{equation}
By our induction hypothesis, $a' \cdot g$ starts and ends with the same letters as $g$ does, this implies that \begin{equation*}
f_{j_0} \cdot (a' \cdot (g \wprod f^{-1}))= f_{j_0} \wprod (a' \cdot g) \wprod f^{-1}.
\end{equation*}
We combine the above with \eqref{thiseq:12:feb3:2026} to compute \eqref{thiseq:11:feb3:2026} as \begin{equation*}
w'_1 \wprod f' \wprod f_{j_0} \wprod (a' \cdot g) \wprod f^{-1} \wprod h
\end{equation*}
we now define $a := a' \cdot g$ to satisfy conditions (ii)--(iii). \textbf{This covers all possibilites when $w''_{N+1} = h \cdot w'_{N+1}$ is the trivial word}. \end{subcase}
 \end{case}

\begin{case}
If $w''_{N+1} = h \cdot w'_{N+1}$ is non-trivial, we compute $f^{-1} \cdot w''_{N+1}$. Since $y_{j_0} \not \in \lett(w''_{N+1})$, we can find some integer $j'''$ such that \begin{equation} \label{this:eq:13:feb3:2026}
f^{-1} \cdot w''_{N+1} = \prod_{i = k}^{j_0} y^{-1}_i \wprod \prod_{i = j_0-1}^{j'''} y^{-1}_i \wprod w'''_{N+1}
\end{equation}
where $w'''_{N+1}$ is an irreducible final-subword (i.e, a suffix) of $w''_{N+1}$. We are once more concerned with whether $w'''_{N+1}$ is trivial or not. 

\begin{subcase} Assume $w'''_{N+1}$ is the trivial word, and define $f''' = f^{-1} \cdot (h \cdot w'_{N+1}) \cdot (f'' \wprod f_{j_0}) = (f^{-1} \cdot w''_{N+1}) \cdot (f'' \wprod f_{j_0})$. We then compute $f'''$ to note that: \begin{equation} \label{this:ftriprime:feb2026}
 f''' = (\prod_{i = k}^{j_0} y^{-1}_i \wprod \prod_{i = j_0-1}^{j'''} y^{-1}_i) \cdot (f'' \wprod f_{j_0}) = (\prod_{i = k}^{j_0} y^{-1}_i \wprod \prod_{i = j_0-1}^{j'''} y^{-1}_i) \cdot (\prod_{i = j'}^{j_0 - 1} y_i \wprod \prod_{i = j_0}^{k} y^{-1}_i )
\end{equation}
If $j''' \neq j'$, then the above product undergoes no cancelations at all (it is simply concatenated) as the letter $y_{j'''}$ would differ from $y_{j'}$. This now allows us to compute \eqref{thiseq:11:feb3:2026} as: \begin{equation*}
w'_1 \wprod f' \wprod f_{j_0} \wprod a' \wprod f''' \wprod g \wprod f^{-1} \wprod h
\end{equation*}
it now suffices to set $a := a' \wprod f''' \wprod g$ to satisfy conditions (i)---(iii). 
\end{subcase}

If we assume that $j''' = j'$, then $f'''$ becomes the trivial word as per \eqref{this:ftriprime:feb2026}. We then compute \eqref{thiseq:11:feb3:2026} as \begin{equation*}
(w'_1 \wprod f' \wprod f_{j_0} \wprod a') \cdot (g \wprod f^{-1} \wprod h) = w'_1 \wprod f' \wprod f_{j_0} \wprod (a' \fprod g) \wprod f^{-1} \wprod h
\end{equation*}
where the last equality holds since $a' \cdot g$ ends with the same letter as $g$ does (by induction hypothesis). We then set $a:= a' \cdot g$ to complete this case. \textbf{This covers all possibilites when $w'''_{N+1}$ is trivial}.

\begin{subcase}
We now assume that $w'''_{N+1}$ is \textbf{not trivial}. Let us now go back to the construction of $w'''_{N+1}$: which is that of a final-subword of $w''_{N+1}$. For this reason, $w'''_{N+1}$ ends with the same letter as $w''_{N+1}$ does. Now, recall that $w''_{N+1} = h \cdot w'_{N+1}$ itself is a \textbf{non-empty} final subword of $w'_{N+1}$, and therefore $w''_{N+1}$ ends with the same letter that $w'_{N+1}$ does. According to \eqref{thiseq:11:feb3:2026}, the product $w'_{N+1} \wprod f''$ is a true concatenation (with no possible further cancelations), since $w'''_{N+1}$ ends with the same letter as $w'_{N+1}$. Thereby, $w'''_{N+1} \cdot f'' = w'''_{N+1} \wprod f''$ holds. We now we employ  \eqref{this:eq:13:feb3:2026} to deduce \begin{equation*}
f^{-1} \cdot h \cdot w'_{N+1} \cdot f'' = (f^{-1} \cdot w''_{N+1}) \cdot f'' = (\prod_{i = k}^{j_0} y^{-1}_i \wprod \prod_{i = j_0-1}^{j'''} y^{-1}_i \fprod w'''_{N+1}) \cdot f''= \prod_{i = k}^{j'''} y^{-1}_i \wprod w'''_{N+1} \wprod f''.
\end{equation*}
Since $w''_{N+1} = h \cdot w'_{N+1}$, we are now able to compute \eqref{thiseq:11:feb3:2026} as \begin{equation*}
(w'_1 \wprod f' \wprod f_{j_0} \wprod a') \wprod (\prod_{i = k}^{j'''} y^{-1}_i \wprod w'''_{N+1} \wprod f'') \wprod (f_{j_0} \wprod g \wprod f^{-1} \wprod h)
\end{equation*}
The concatenation between the first two terms above holds since $a'$ ends with the same letter $g$ does (which is in the set $X$). We now are able to define $a:= a' \wprod \prod_{i = k}^{j'''} y^{-1}_i \wprod w'''_N \wprod f'' \wprod f_{j_0} \wprod g$ to satisfy conditions (i)---(iii).
\end{subcase}
Our proof is now concluded, as we have covered all possible cases.
\end{case}
\end{proof}

\begin{corollary} \label{this:claim100:oct292024}
Under Hypothesis \ref{this:master:hypothesis} for a Setting \ref{this:master:set}$(X,Y,g,h,k)$, assume there exists some \by $1 \leq N \leq n$ \ey for which \by $v_1 = v_2 = \dots = v_N$. \ey Then, the product $w_1 \cdot v_1 \cdot w_2 \cdot \cdots \cdot w_{N} \cdot v_N \cdot w_{N+1}$ is not a word in $F(X)$. 
\end{corollary}

\begin{proof}
Here we consider two cases. 
\begin{case} \label{begin:case1:thufeb5of2026}
If $v_1 = v_2 = \dots = v_N = g_0$.

In this case, we directly apply Lemma \ref{this:claim:oct292024} to calculate \begin{equation*}
w_1 \cdot v_1 \cdot w_2 \cdot \cdots \cdot w_{N} \cdot v_N = w'_1 \wprod f' \wprod f_{j_0} \wprod a \wprod f^{-1} \wprod h
\end{equation*} 
in \textbf{irreducible} form. Since $y_{j_0} \in \lett(f^{-1}) \setminus \lett(w_{N+1})$ (by hypothesis), the product $(h \cdot w_{N+1})$ also does not contain the letter $y_{j_0}$. Thereby, there exists an irreducible final subword (a suffix) $u_{N+1}$ of $h \cdot w_{N+1}$, and an integer $j \in \{j_0, \dots, n\}$ such that \begin{equation}
f^{-1} \cdot (h \cdot w_{N+1}) = f^{-1}_{j_0} \wprod \prod_{i = j}^{j_0-1} y_i \wprod u_{N+1}
\end{equation} 
in irreducible form. We conclude that \begin{equation*}
w_1 \cdot v_1 \cdot w_2 \cdot \cdots \cdot w_{N} \cdot v_N \cdot w_{N+1} = w'_1 \wprod f' \wprod f_{j_0} \wprod a \wprod f^{-1}_{j_0} \wprod \prod_{i = j}^{j_0-1} y_i \wprod u_{N+1}
\end{equation*}
in irreducible form. In particular, $f_{j_0}$ is a full irreducible sub-word above, ensuring that the letter $y_{j_0}$ appears in $w_1 \cdot v_1 \cdot w_2 \cdot \cdots \cdot w_{N} \cdot v_N \cdot w_{N+1}$. We conclude that it is therefore not a word in $F(X)$. 
\end{case}
\begin{case}
If $v_1 = v_2 = \dots = v_N = g^{-1}_0$.
\end{case}
Recall that $g^{-1}_0 = h^{-1} \wprod f \wprod g^{-1} \wprod f^{-1}$ in irreducible form. We are going to re-index words in the following way: \begin{itemize}
\item Let $\hat{g_0} = f \wprod g^{-1} \wprod f^{-1} \wprod e$, 
\item for each $i = 1,\dots,N$ we let $\hat{w}_i = w_i \cdot h^{-1}$, 
\item we let $\hat{w}_{N+1} = w_{N+1}$.
\end{itemize}
We now note that $\hat{g_0}$ is the special word considered in \eqref{this:special:word:feb5:2026} of the Setting \ref{this:master:set}$(X,Y,g^{-1},e,k)$ (where $e$ is the empty word). Since $h \in F(X)$, for each $i = 1,\dots,n+1$ the word $\hat{w}_i$ does not contain the letter $y_{j_0}$. 
We then note the following: $\hat{v_1} = \hat{v_2} = \dots = \hat{v_N} = \hat{g_0}$ and more over, \begin{equation}
\hat{h^*} = \hat{w}_1\fprod \hat{v}_1\fprod \hat{w}_2\fprod \hat{v}_2\fprod\cdots\fprod \hat{w}_N \fprod \hat{v}_N\fprod \hat{w}_{N+1}
\end{equation}
satisfies all conditions of Hypothesis \ref{this:master:hypothesis}. We are now under the hypotheses of Lemma \ref{this:claim:oct292024}, and then of the hypotheses of Case \ref{begin:case1:thufeb5of2026} as well. We proceed exactly as in this case to deduce that $y_{j_0}$ is a letter which appears in the product \begin{equation*}
w_1 \cdot v_1 \cdot w_2 \cdot \cdots \cdot w_{N} \cdot v_N \cdot w_{N+1} = \hat{w}_1 \cdot \hat{v}_1 \cdot \hat{w}_2 \cdot \cdots \cdot \hat{w}_{N} \cdot \hat{v}_N \cdot \hat{w}_{N+1}.
\end{equation*}
In particular, this word cannot be in $F(X)$. This concludes the proof of this corollary.
\end{proof}
\ey

\by
\begin{lemma} \label{this:new2024:cancelation:lemma}
Assume Hypothesis \ref{this:master:hypothesis} for a Setting \ref{this:master:set}$(X,Y,g,h,k)$. Then, $h^* = w^*$ where \begin{equation} \label{this:eq100:oct292024}
w^* = w_1 \cdot e \cdot w_2 \cdot \cdots w_n \cdot e \cdot w_{n+1}
\end{equation}
is the word obtained from $h^*$ \eqref{this:eq101:oct292024} with all the words $v_1,\dots, v_n$ replaced by $e$ (the empty word).
\end{lemma}
\ey
\begin{proof}

The proof shall be done by contradiction, using the principle of minimal counterexample. Namely, assume there exists a word $h^* = w_1\fprod v_1\fprod w_2\fprod v_2\fprod\cdots\fprod w_n\fprod v_n\fprod w_{n+1}$  satisfying conditions \by (i)--(iii) of Hypothesis \ref{this:master:hypothesis}, \ey and for which $h^* \neq w^*$ where $w^*$ is as in \eqref{this:eq100:oct292024}. Assume then, that the counterexample $h^*$ is of the minimal possible $n \in \N$ for the indices above. Since $h^* \in F(X)$, the word $h^*$ cannot contain any of the letters embedded in $f$ in its irreducible form (as $\lett(f) \subseteq Y \setminus X$). This implies that the instance of $f$ in the irreducible form of $v_1$ must cancel as a result of the computation of the product \eqref{this:eq101:oct292024}. Without loss of generality, let us assume that $v_1 = g_0$. \by By Corollary \ref{this:claim100:oct292024}, it is \ey impossible for the equality $v_1 = v_2 = \dots = v_n$ to hold. This implies that the instance of $f$ in $v_1$ may only be canceled by some existing $v_j$ in the family $\{v_2,\dots,v_n\}$ for which $v_j = g^{-1}_0$; let us fix this word $v_j$. Let us now expand $v_1$ and $v_j$ in the product from \eqref{this:eq101:oct292024} up to the $j$-th term, \by and denote the resulting word by $h^*_j$: \ey
\begin{equation*}
h^*_j := w_1\fprod (f \wprod g \wprod f^{-1} \wprod h) \fprod w_2\fprod v_2\fprod\cdots\fprod w_j \fprod (h^{-1} \wprod f \wprod g^{-1} \wprod f^{-1})
\end{equation*}
and let us now expand $f$ and $f^{-1}$ in their full irreducible representations:
\begin{equation} \label{this:eq10:oct30:2024}
w_1\fprod (\prod_{i=1}^k y_i \wprod g \wprod \prod_{i=k}^1 y^{-1}_i \wprod h) \fprod w_2\fprod v_2\fprod\cdots\fprod w_j \fprod (h^{-1} \wprod \prod_{i=1}^k y_i \wprod g^{-1} \wprod \prod_{i=k}^1 y^{-1}_i )
\end{equation}
Since the instance of $f$ in $v_1$ must cancel with the instance of $f^{-1}$ in $v_j$, in particular the letter $y_1$ itself may only be canceled by the letter $y^{-1}_i$ appearing right-most in \eqref{this:eq10:oct30:2024}. \by This implies that, in order for the left-most letter $y_1$ of $v_1$ to meet and cancel with the right-most letter $y^{-1}_1$ of $v_j$, the entire intermediate product trapped between them must completely collapse to the empty word. \ey One necessarily has
\begin{equation} \label{this:eq:121:oct302024}
(\prod_{i=1}^k y_i \wprod g \wprod \prod_{i=k}^1 y^{-1}_i \wprod h) \fprod w_2\fprod v_2\fprod\cdots\fprod w_j \fprod (h^{-1} \wprod \prod_{i=1}^k y_i \wprod g^{-1} \wprod \prod_{i=k}^1 y^{-1}_i ) = y_1 \cdot e \cdot y^{-1}_1 = e.
\end{equation}
We now have that, in fact $h^*_j = w_1 \cdot e = w_1$. Let us now use this equality to simplify $h^*$: 
\begin{equation} \label{this:eq5:oct302024}
h^* = h^*_j \cdot w_{j+1} \cdot v_{j+1} \cdots w_n \cdot v_n \cdot w_{n+1} = (w_1 \cdot w_{j+1}) \cdot v_{j+1} \cdots w_n \cdot v_n \cdot w_{n+1}. 
\end{equation}
We now re-index the terms in \eqref{this:eq5:oct302024} as follows: set $w'_1 = w_1 \cdot w_{j+1}$, $w'_{n+1-j} = w_{n+1}$, $v'_1 = v_{j+1}$ and for each $i \in \{2,\dots,n-j\}$, set $v'_i = v_{j+i}$ and $w'_i = w_{j+1}$. We now rewrite the product in \eqref{this:eq5:oct302024} as \begin{equation} \label{this:eq:11:oct30:2024}
h^* = w'_1 \cdot v'_1 \cdot w'_2 \cdot v'_2 \cdot \cdots \cdot w'_{n-j} \cdot v'_{n-j} \cdot w'_{n+1-j}.
\end{equation}
\by We now note the following: $y_{j_0} \not \in \lett(w'_1) = \lett(w_1 \cdot w_{j+1})$ holds, as neither $w_1$ nor $w_{j+1}$ contain $y_{j_0}$ as a letter of their irreducible forms. \ey Next, for each $i = 1,\dots,n-j$ the word $v'_{i} = v_{j+1}$ is an instance of either $g_0$ or $g^{-1}_0$, and $h^* \in F(X)$ holds by our original assumptions. Since $n-j < n$, it suffices to prove that \eqref{this:eq:11:oct30:2024} is a counter-example to the conclusion of our lemma to arrive to a contradiction. \begin{claim}
Consider the word $$v^* = w'_1 \cdot e \cdot w'_2 \cdot \cdots \cdot w'_{n-j} \cdot e \cdot w'_{n+1-j}$$
obtained from  \eqref{this:eq:11:oct30:2024} by replacing all instances of the $v'_j$'s with $e$. Then, $h^* = v^* = w^*$.
\end{claim}
\begin{proof}
\by 
The equality $h^* = v^*$ holds by the following: since $n - j < n$, then the sequence as written in \eqref{this:eq5:oct302024} is not a counter-example to our lemma, so we may apply it. An application of that lemma implies that $h^* = v^*$, where $v^*$ is simply the expression in \eqref{this:eq:11:oct30:2024} by replacing all $v'_i$ by the empty word $e$. Let us now prove that $v^* = w^*$.
\ey

Re-expand the definition of the $w'_j$'s as follows: \begin{equation} \label{this:eq90:oct302024} v^* = (w_1 \cdot e \cdot w_{j+1}) \cdot e \cdot w_{j+2} \cdot e \cdots \cdot w_{n} \cdot e \cdot w_{n+1}.
\end{equation}
For the $e$ term in-between $w_1$ and $w_{j+1}$ we now expand it as in \eqref{this:eq:121:oct302024} to deduce that \by
\begin{equation} \label{this:eq1:feb6:2026}
v_1 \cdot w_2 \cdot v_2 \cdot \cdots \cdot w_j \cdot v_j = e.
\end{equation}
Since $j < n$, our assumption allows us to apply the our lemma to the above sequence as it cannot be a counterexample (due to minimality). To apply it, we define $\hat{w}_1 = e$, and set $\hat{w}_i = w_i$ and $\hat{v}_i = v_i$ for all $i = 1,\dots,j$. From \eqref{this:eq1:feb6:2026} we have \begin{equation}
e = v_1 \cdot w_2 \cdot v_2 \cdot \cdots \cdot w_j \cdot v_j  = \hat{w}_1 \cdot \hat{v}_1 \cdot \hat{w}_2 \cdot \hat{v}_2 \cdot \cdots \cdot \hat{w}_j \cdot \hat{v}_j.
\end{equation}
Since $e \in F(X)$, and the above sequence satisfies conditions (i)-(iii) of our lemma. The minimality of $n$ and $j < n$ allows us to conclude that 
\begin{equation*}
\hat{w}_1 \cdot \hat{v}_1 \cdot \hat{w}_2 \cdot \hat{v}_2 \cdot \cdots \cdot \hat{w}_j \cdot \hat{v}_j = \hat{w}_1 \cdot e \cdot \hat{w}_2 \cdot e \cdot \cdots \cdot \hat{w}_{j+1} 
\end{equation*} 
where all $\hat{v}_i$'s have been replaced by the empty word. Since $\hat{w}_i = w_i$ holds for all $i=1,\dots,j$ we actually proved that \begin{equation} \label{this:eq2:feb6:2026}
e = e \cdot v_1 \cdot w_2 \cdot v_2 \cdot \cdots \cdot w_j \cdot v_j = w_2 \cdot e \cdot w_3 \cdot e \cdot \cdots \cdot w_{j-1} \cdot e \cdot w_j \cdot e.
\end{equation}
We now replace \eqref{this:eq2:feb6:2026} in the first $e$-term of \eqref{this:eq90:oct302024} to see that \begin{equation*}
v^* = (w_1 \cdot (e \cdot w_2 \cdot e \cdot \cdots \cdot w_j \cdot e) \cdot w_{j+1}) \cdot e \cdot w_{j+2} \cdot e \cdots \cdot w_{n} \cdot e \cdot w_{n+1}
\end{equation*}
This now proves that $v^* = w^*$ as per \eqref{this:eq100:oct292024}. 
\ey
\end{proof}
\by
By our previous claim we now have that $h^* = v^* = w^*$. However, this is a contradiction: since we assumed we had a counter-example to our lemma, then $h^* \neq w^*$ was supposed to hold. 
We conclude that no such counter-example exists, proving our lemma.
\ey
\end{proof}

\by
\begin{corollary}
\label{workingman:lemma:7.2}
Take a Setting \ref{this:master:set} $(X,Y,g,h,k)$. Assume we are given a sequence $a_1,a_2,\dots,a_m\in F(Y)$ such that $a_1\fprod a_2\fprod\dots\fprod a_m \in F(X)$ holds. If for each $l = 1,\dots,m$ the inclusion $y_{j_0} \in \lett(a_l)$ holds if and only if $a_l \in \grp{g_0} \setminus \{e\}$, then
\begin{equation}
\label{eq:33:z}
a_1\fprod a_2\fprod\dots\fprod a_m
=
\eta(a_1)\fprod \eta(a_2)\fprod\dots\fprod \eta(a_m).
\end{equation}  
\end{corollary}
\ey

\begin{proof}
Define the set
\begin{equation}
\label{eq:def:L}
L=\{l=1,\dots,m: a_l \in \grp{g_0} \setminus \{e\}\}.
\end{equation}
If $L = \emptyset$ holds, then for each $l = 1,\dots,m$ we have $a_l = \eta(a_l)$ by the definition of $\eta$, allowing us to deduce \eqref{eq:33:z}. \by From now on, we always assume that $L \neq \emptyset$. \ey Note the following: for any $l \in L$ the inclusion  $a_l \in \grp{g_0}$ holds, therefore, there exists some $q_l \in \Z$ such that $a_l = g_0^{q_l}$. Let $\delta_l = q_l / |q_l| \in \{-1,1\}$ be the sign of $q_l$ \by (naturally, here $|q_l|$ is the absolute value of $q_l$). \ey Let us rewrite $a_l$ as the following product: \by \begin{equation} \label{this:eq:oct312024}
a_l = g_0^{q_l} = \prod_{j = 1}^{|q_l|} (g_0^{\delta_l} \cdot e) = g_0^{\delta_l} \cdot e  \cdot g_0^{\delta_l} \cdot \cdots \cdot e \cdot g_0^{\delta_l} \cdot e,
\end{equation}
\ey
\by where $g_0^{\delta_l}$ appears $|q_l|$-many times. \ey We now apply the map $\eta$ \by term-by-term to \ey both sides of the above to note \begin{equation} \label{this:eq2:oct312024}
e = \eta(a_l) = \prod_{j = 1}^{|q_l|} (\eta(g_0^{\delta_l}) \cdot \eta(e)) = \eta(g_0^{\delta_l}) \cdot \eta(e)  \cdot \eta(g_0^{\delta_l}) \cdot \cdots \cdot \eta(e) \cdot \eta(g_0^{\delta_l}) \cdot \eta(e).
\end{equation}

Let us take an increasing indexing $L = \{t_i :  1 \leq i \leq N\}$ of the set $L$ where $N \in \N$. 

\by
Our proof will proceed as follows: we shall group the sequence into ``blocks'' separated by the non-trivial powers of $g_0$. By substituting these powers with their explicit representations (as sequences of $g_0^{\pm 1}$ interspersed with $e$), we can align the product into a format that satisfies Hypothesis \ref{this:master:hypothesis}.
\begin{claim} \label{this:key:relabling:claim:feb82026}
The product $a_1\fprod a_2\fprod\dots\fprod a_{t_N}$ can be re-written as a new product of the form $w_1 \fprod v_1 \fprod w_2 \fprod v_2 \fprod \cdots \fprod w_n \fprod v_n$ of words in $F(Y)$ (for some $n \in \N$) with the following properties: \begin{itemize}
\item[(i)] $y_{j_0} \not \in \lett(w_i)$ for all $i \in \{1, \dots, n+1\}$,
\item[(ii)] $v_i \in \{g_0,g^{-1}_0\}$ for each $i \in \{1, \dots,n\}$; and
\item[(iii)] the equalities \begin{equation} \label{this:master:product:eq:feb8:2026}
\prod_{t \leq t_N} a_t = w_1 \fprod \left( \prod_{0<i<n} v_i \cdot w_{i+1} \right) \cdot v_{n}\mbox{; and}
\end{equation}
\begin{equation}
\prod_{t \leq t_N} \eta(a_t) = w_1 \fprod \left( \prod_{0<i<n} e \cdot w_{i+1} \right) \cdot e. 
\end{equation}
hold.
\end{itemize}
\end{claim}
\begin{proof}

Associate factors of \eqref{this:master:product:eq:feb8:2026} to obtain:
\begin{equation} \label{this:eq:feb7}
(\prod_{t < t_1} a_t) \fprod a_{t_1} \fprod (\prod_{t_1 < t < t_2} a_t) \fprod a_{t_2} \fprod \cdots (\prod_{{t_{N-1}}< t < t_N} a_t) \fprod a_{{t_N}}. 
\end{equation}

Let us note the following: by the definition of $L$ we have that \begin{equation} \label{this:noyj0:equation}
y_{j_0} \not \in \lett(\prod_{t < t_1} a_t) \cup \bigcup_{1 \leq i < N} \lett(\prod_{{t_{i}}< t < t_{i+1}} a_t)).
\end{equation}
We construct the sequence $w_1 \cdot v_1 \cdot w_2 \cdot v_2 \cdot \cdots \cdot w_n \cdot v_n$ by following the visual pattern given in \eqref{this:eq:feb7}. Our construction is done by induction over $N$ (the cardinality of $L$). 

Starting with the first index, $t_1$, we define \begin{equation} \label{this:eq10:feb07:2026}
w_1 = \begin{cases} e \mbox{ if } t_1 = 1,\mbox{ or} \\
\prod_{t < t_1} a_t \mbox{ otherwise}.
\end{cases}
\end{equation}
Note that $y_{j_0} \not \in \lett(w_1)$ by \eqref{this:noyj0:equation}. Let us now set $v_1 = g_0^{\delta_{t_1}}$, $w_2 = e$, $v_2 = g_0^{\delta_{t_1}}$, ... , $v_{|q_{t_1}|} = g_0^{\delta_{t_1}}$ %, $w_{|q_{t_1}|+1} = e$ 
such that: \begin{equation} \label{this:eq11:feb07:2026}
a_{t_1}= g_0^{q_{t_1}} = \left[ \prod_{0 < i < |q_{t_1}|} (g_0^{\delta_{t_1}} \cdot e) \right] \fprod  g_0^{\delta_{t_1}}=  \left[ \prod_{0 < i < |q_{t_1}|} (v_i \cdot w_{i+1}) \right] \fprod v_{|q_{t_1}|}.
\end{equation}

Let us now apply the map $\eta$ to the sequence $\prod_{t \leq t_1} a_t$ term-by-term, and use the respective expressions in \eqref{this:eq10:feb07:2026} and \eqref{this:eq11:feb07:2026} to deduce that: \begin{equation*}
\prod_{t < t_1} \eta(a_t) = \prod_{t < t_1} a_t = w_1,
\end{equation*}
and from \eqref{this:eq2:oct312024} we then see
\begin{equation*}
e = \eta(g_0^{\delta_{t_1}}) = \eta(a_{t_1}) = \left[ \prod_{0 < i < |q_{t_1}|} (\eta(g_0^{\delta_{t_1}})  \cdot \eta(e)) \right] \fprod \eta(g_0^{\delta_{t_1}}). 
\end{equation*}
This now proves that \begin{equation*}
\prod_{t \leq {t_1}} \eta(a_t) = w_1 \fprod \left[ \prod_{0 < i < |q_{t_1}|} (\eta(v_i) \fprod \eta(w_{i+1})) \right] \fprod \eta(v_{|q_{t_1}|}) = w_1 \fprod \left[ \prod_{0 < i < |q_{t_1}|} (e \fprod w_{i+1}) \right] \cdot e.   
\end{equation*}
Let us now state the induction hypothesis for a given $p < N$. Let us define the integers \begin{equation} \label{this:subI:def}
I_p = \sum_{1 \leq i \leq p} |q_{t_i}|, \mbox{ and } I_{p+1} = \sum_{1 \leq i \leq p+1} |q_{t_i}|.
\end{equation}
Note that the above integers are used to count \emph{how many times} the terms $v_i$ will appear, as we shall see below. Let us now assume that for this $p < N$ we have  \begin{equation} \label{this:eq:46:feb8:2026}
\prod_{t \leq {t_p}} a_t = w_1 \fprod \left[ \prod_{0 < i < I_p} (v_i \fprod w_{i+1}) \right] \fprod v_{I_p},
\end{equation}
where the $w_i$'s and $v_i$'s satisfy conditions (i)--(ii); and
\begin{equation} \label{this:eq:47:feb8:2026}
\prod_{t \leq {t_p}} \eta(a_t) = w_1 \fprod \left[ \prod_{0 < i < I_p} (e \fprod w_{i+1}) \right] \fprod e.
\end{equation}
We now proceed by defining \begin{equation} \label{this:eq:48:feb8:2026}
w_{{I_p}+1} = \begin{cases} e \mbox{ if } t_{p}+1 = t_{p+1} \in L,\mbox{ or} \\
\prod_{t_p < t < t_{p+1}} a_t \mbox{ otherwise}.
\end{cases}
\end{equation}
Once more, $y_{j_0} \not \in \lett(w_{{I_p}+1})$ by \eqref{this:noyj0:equation}. Let us now set $v_{{I_p}+1} = g_0^{\delta_{t_{p+1}}}$, $w_2 = e$, $v_2 = g_0^{\delta_{t_{p+1}}}$,..., $v_{I_p+|q_{t_{p+1}}|} = v_{I_{p+1}} = g_0^{\delta_{t_{p+1}}}$ such that \begin{equation}  \label{this:eq:49:feb8:2026}
a_{t_{p+1}} = g_0^{\delta_{q_{t_{p+1}}}} = \left[ \prod_{0 < i < |q_{t_{p+1}}|} (g_0^{\delta_{t_{p+1}}} \cdot e) \right] \fprod  g_0^{\delta_{t_{p+1}}}=  \left[ \prod_{0 < i < |q_{t_{p+1}}|} (v_{{I_p}+i} \cdot w_{{(I_p)}+i+1}) \right] \fprod v_{I_{p+1}}.
\end{equation}
We now employ \eqref{this:eq:46:feb8:2026} from our induction hypothesis and \eqref{this:eq:48:feb8:2026} 
to see that 
\begin{equation} \label{this:eq:50:feb8:2026}
\prod_{t < t_{p+1}} a_t = \prod_{t \leq {t_p}} a_t \fprod  \prod_{t_p < t < {t_{p+1}}} a_t = \left( w_1 \fprod \left[ \prod_{0 < i < I_p} (v_i \fprod w_{i+1}) \right] \fprod v_{I_p} \right) \fprod w_{{(I_p)}+1}.
\end{equation}
Let us now apply the map $\eta$ to the sequence $\prod_{t \leq t_{p+1}} a_t$ term-by-term by using the respective expressions in \eqref{this:eq:49:feb8:2026} and \eqref{this:eq:50:feb8:2026}.  %to deduce that

First, we use \eqref{this:eq:47:feb8:2026} from our induction hypothesis to obtain: \begin{equation} \label{this:eq:59:feb8:2026}
\prod_{t < t_{p+1}} \eta(a_t) = \prod_{t \leq {t_p}} \eta(a_t) \fprod  \prod_{t_p < t < {t_{p+1}}} \eta(a_t) = \left( w_1 \fprod \left[ \prod_{0 < i < I_p} (e \fprod w_{i+1}) \right] \fprod e \right) \fprod w_{{(I_p)}+1}.
\end{equation}

We now employ \eqref{this:eq2:oct312024} on \eqref{this:eq:49:feb8:2026} to obtain
\begin{equation} \label{this:eq:60:feb8:2026}
e = \eta(g_0^{\delta_{q_{t_{p+1}}}}) = 
\eta(a_{t_{p+1}}) = \left[ \prod_{0 < i < |q_{t_{p+1}}|} (\eta(v_{{I_p}+i}) \cdot \eta(w_{{(I_p)}+i+1})) \right] \fprod \eta(v_{I_{p+1}})
\end{equation}
where,
\begin{equation} \label{this:eq:61:feb8:2026}
\left[ \prod_{0 < i < |q_{t_{p+1}}|} (\eta(v_{{I_p}+i}) \cdot \eta(w_{{(I_p)}+i+1})) \right] \fprod \eta(v_{I_{p+1}}) = \left[ \prod_{0 < i < |q_{t_{p+1}}|} (e \cdot w_{{(I_p)}+i+1}) \right] \fprod e.
\end{equation}

Combining \eqref{this:eq:59:feb8:2026}, \eqref{this:eq:60:feb8:2026} and \eqref{this:eq:61:feb8:2026}, we are now able to deduce that: \begin{equation}
\prod_{t \leq {t_{p+1}}} \eta(a_t) = w_1 \fprod \left[  \prod_{0 < i < I_{p+1}} (e \fprod w_{i+1}) \right] \fprod e 
\end{equation}
This concludes the induction step. To conclude the proof of this claim, we employ the construction above up to the integer $n := I_N$ as defined in \eqref{this:subI:def}.
\end{proof}
\ey
Let us define the word \begin{equation} \label{last:important:term:feb82026}
w_{n+1} = \begin{cases} e \mbox{ if } t_{N} = a_m \\
\prod_{t_{N} < t} a_t \mbox{ otherwise}.
\end{cases}
\end{equation}
By \eqref{this:noyj0:equation}, we have $y_{j_0} \not \in \lett(w_{n+1})$, and more over \begin{equation} \label{pleasegod:feb072026}
w_{n+1} = \eta(w_{n+1}) = \prod_{t_{N} < t} \eta(a_t). 
\end{equation}
\by Set $n := I_N$ as in \eqref{this:subI:def}. \ey Let us now employ Claim \ref{this:key:relabling:claim:feb82026} together with \eqref{last:important:term:feb82026} re-write the product $a_1\fprod a_2\fprod\dots\fprod a_{m}$ as \begin{equation}
h^* = \left(\prod_{t \leq t_N} a_t\right) \fprod \left(\prod_{t_{N} < t} a_t\right) = w_1 \fprod \left( \prod_{0<i<{n+1}} v_i \cdot w_{i+1} \right) 
\end{equation}
By our Claim \ref{this:key:relabling:claim:feb82026} (and the construction of $w_{n+1}$ and $h^*$) we now satisfy \textbf{all} conditions of Hypothesis \ref{this:master:hypothesis}; and so we may apply Lemma \ref{this:new2024:cancelation:lemma}. Following its use, we define the word $w*$, and note that \begin{equation*}
h^* = w^* = w_1 \fprod \left( \prod_{0<i<{n+1}} e \cdot w_{i+1} \right) = \eta(w_1) \fprod \left( \prod_{0<i<{n+1}} \eta(v_i) \cdot \eta(w_{i+1}) \right).
\end{equation*}
Following our Claim \ref{this:key:relabling:claim:feb82026} and \eqref{pleasegod:feb072026} we deduce from the above that \begin{equation*}
\prod_{t \leq m} \eta(a_t) = h^* = \prod_{t \leq m} a_t.
\end{equation*}
This concludes the proof.
\end{proof}

\section{Finite neighbourhood system extension via foreign word conjugation}
\label{extension:section}

\by
First, we begin with a lemma about intermediate enrichments:

\begin{lemma} \label{this:new:remark:20oct24}
Let $X$ be a set, and assume $\nice{U} = \{U_i : i \leq n\}$ satisfies conditions (a), (b) and (c) of Lemma \ref{enrichment:lemma}. We let $B',B \subseteq F(X)$ and $\nice{V'},\nice{V}$ be sets such that: \begin{itemize}
\item[(i)] $B' \subseteq B$, and both $B$ and $B'$ are symmetric subsets of $F(X)$,
\item[(ii)] $\nice{V'} = \{V_i: i \leq n\}$ denotes the $B'$-enrichment of $\nice{U}$,
\item[(iii)] $\nice{V} = \{V_i: i \leq n\}$ denotes the $B$-enrichment of $\nice{U}$,
\end{itemize}
Then, the inclusion \by $V'_i \subseteq V_i$ holds for all $i \leq n$. Moreover, \ey if $h = a_1 \cdot a_2 \cdots a_m$ is a canonical representation for some $h \in V'_i$, then it is also a canonical representation for $h \in V_i$. 
\end{lemma}

\begin{proof}
Our proof is by finite-reverse induction on $n$ (the length of the sequence $\nice{U}$). 

By \eqref{eq:def:Vn}, we have that $V'_n = U_n \cup B'$, and $V_n = U_n \cup B$. This implies that $V'_n \subseteq V_n$ holds by our hypothesis on $B$ and $B'$. Assume that for some $j \in \{1,\dots,n\}$ the inclusion $V'_j \subseteq V_j$ holds. Our goal is now to prove that $V'_{j-1} \subseteq V_{j-1}$ holds. 

If we expand \eqref{eq:def:Vi} for $V'_{j-1}$ and $V_{j-1}$ we have: \begin{equation*}
V'_{j-1}=U_{j-1} \cup\bigcup_{x\in \dash{X}} x\fprod V'_{j}\fprod V'_{j}\fprod x^{-1} \subseteq 
V'_{j-1} \subseteq U_{j-1} \cup\bigcup_{x\in \dash{X}} x\fprod V_{j}\fprod V_{j}\fprod x^{-1} = 
V_{j-1},
\end{equation*}
where the middle inclusion holds by induction hypothesis. This concludes our induction.

The last claim about presevation of the canonical representations is proved in the same way. For $h \in V'_n$ we have that $h = a_1$ itself is a canonical representation, where $h = a_1 \in V_n$ by the previous proof. Thereby, any canonical representation of a word in $V'_n$ is a canonical representation for the same word in $V_n$. Assume for $j \in \{1,\dots,n\}$ we have proven that if $h = a_1 \cdot a_2 \cdot \cdots \cdot a_m$ is a canonical representation for $h \in V'_{j}$, then it is also a canonical representation for $h$ considered in $V_j$. 

Let $h \in V'_{j-1}$ be arbitrary. Our goal is prove that all possible canonical representations of $h$ with respect to $V'_{j-1}$ are also canonical representations of $h$ with respect to $V_{j-1}$. Following the construction of canonical representations, we consider the two cases according to Definition \ref{canonical:representation:definition}.

{\sl Case 1\/}. If $h = a_1$ is a canonical representation of $h$ with respect to $V'_{j-1}$, then $h = a_1 \in V'_{j-1} \subseteq V_{j-1}$ is a canonical representation of $h$. 

{\sl Case 2\/}. Assume the canonical representation of $h$ with respect to $V'_{j-1}$ is of the form
\begin{equation} \label{this:eq:jan20of2026}
h=x\fprod 
b_1\fprod b_2\fprod\cdots\fprod b_{m_1}
\fprod
c_1\fprod c_2\fprod\cdots\fprod c_{m_2}
\fprod x^{-1},
\end{equation}
for a suitable $x \in \bar{X}$ and $u,v \in V'_{j}$ with canonical representations \begin{equation*}
u=b_1\fprod b_2\fprod\cdots\fprod b_{m_1}
\
\text{ and }
\ 
v=c_1\fprod c_2\fprod\cdots\fprod c_{m_2}
\end{equation*}
with respect to $V'_j$. By induction hypothesis, the representations $u = b_1\fprod b_2\fprod\cdots\fprod b_{m_1}$ and $v = c_1\fprod c_2\fprod\cdots\fprod c_{m_2}$ are also canonical representations of $u$ and $v$ respectively when considered in $V_{j}$.

By Case 2 of Definition \ref{canonical:representation:definition}, the representation of $h$ given in \eqref{this:eq:jan20of2026}
is now canonical for $h$ when considered in $V_j$. This concludes our induction, and the proof of this lemma.
\end{proof}

\ey

\begin{corollary}
\label{trivial:enrichment:bound}
Let $X$ and $Y$ be sets such that $X\not=\emptyset$ and $X\subseteq Y$. Let $\nice{U}=\{U_i:i\le n\}$ be a \fns\ for $F(X)$ and let $\nice{V'}=\{V'_i:i\le n\}$ be the \by (cyclic) \ey $\{e\}$-enrichment of $\nice{U}$
in $F(Y)$. Then
\begin{equation}
\label{sum:of:letters:20oct24}
\sum_{l=1}^m |\lett(a_l)|\le |X|\cdot 4^{n-i}
\end{equation}
whenever $i \leq n$
and 
$h=a_1\fprod a_2\fprod\cdots\fprod a_m$ is a
canonical representation of $h\in V'_i$ as in Definition 
\ref{canonical:representation:definition}.
\end{corollary}

\begin{proof}
Recall that the cyclic $(Y \setminus X)$-enrichment of $\nice{U}$ is with respect to the set \by $B = \bigcup_{y \in Y \setminus X} \grp{y}$ \ey according to Definition \ref{def:enrichment}. This implies that \by $B' = \{e\} \subseteq B$. \ey Let $i \leq n$ be arbitrary. By Lemma \ref{this:new:remark:20oct24}, if $h = a_1 \cdot a_2 \cdots a_m$ is a canonical representation of $h \in V'_i$, then it is also a canonical representation of $h \in V_i$ as well. We may now apply Lemma \ref{extension:set:c} to directly deduce \eqref{sum:of:letters:20oct24}.
\end{proof}

\begin{lemma} \label{this:lemma:oct26}
Let $X$ \by be a subset of $Y$ such that \ey $k = | Y \setminus X| \geq 2^{|X| \cdot 4^n}$. 
Take the Setting \ref{this:master:set} $(X,Y,g,h,k)$ for arbitrary $g \in F(X) \setminus \{e\}$ and $h \in F(X)$. Let $g_0 \in F(Y)$ be an arbitrary word containing $f$ as a sub-word of its irreducible form. 
Let $\nice{U} = \{U_i : i \leq n \}$ be a {\fns} on $F(X)$, and let $\nice{V}' = \{V'_i: i \leq n\}$ and $\nice{V} = \{V_i: i \leq n\}$ be its respective $B' = \{e\}$ and $B = \grp{g_0}$ enrichments in F(Y). \by Then,  \begin{equation}
(B \setminus B') \cap (\dash{Y} \cup \bigcup_{i=1}^n V'_i) = \emptyset. 
\end{equation}
\ey
\end{lemma}

\begin{proof}
\by Let us employ the following observation about words on a free group: the word $g_0$ admits an irreducible decomposition of the form $w \wprod v \wprod w^{-1}$, where $w$ is the maximal \emph{initial} subword of $g_0$ for which $w^{-1}$ appears at the end of $g_0$ in its irreducible form. This ensures that $v$ undergoes no cancelations with itself, namely, $v^2 = v \wprod v$ in irreducible form. This implies that, for any $k \in \Z \setminus \{0\}$ the word $g^k_0$ admits the decomposition $w \wprod v^k \wprod w^{-1}$ in irreducible form. In particular, this implies that $\lett(g^k_0) = \lett(g_0)$.
\ey

Given arbitrary $g \in \grp{g_0} \setminus \{e\}$, the word $g$ contains a full copy of $f$ as sub-word in its irreducible form, implying that $|\lett(g)| \geq 2^{|X| \cdot 4^n}$ by construction of $f$ and the previous observation.

Now, let $e \neq h \in \dash{X} \cup \bigcup_{i=1}^n V'_i$ be arbitrary. If $h \in \dash{Y}$, then $|\lett(h)| = 1$ holds, implying that $h \neq g$ for any $g \in \grp{g_0} \setminus \{e\}$ by the above. If $h \in V'_i$ for some $i \in \{1,\dots,n\}$, then Corollary \ref{trivial:enrichment:bound} implies that $\sum_{l=1}^m |\lett(a_l)|\le |X|\cdot 4^{n-i}$ where $h = a_1\fprod a_2\fprod\cdots\fprod a_m$ is a canonical representation of $h \in V'_i$. By Lemma \ref{letters:and:products} we now have \begin{equation*}
|\lett(h)| \leq \sum_{l=1}^m |\lett(a_l)|\le |X|\cdot 4^{n-i} < 2^{|X| \cdot 4^n} \leq |\lett(g)| 
\end{equation*}
where $g \in \grp{g_0} \setminus \{e\}$ is arbitrary. This implies that $h \not \in \grp{g_0} \setminus \{e\}$, proving our assertion.
\end{proof}

\begin{lemma} \label{this:corollary:oct312024:main}
\by Let $X$ be a subset of $Y$ such that $k = | Y \setminus X| \geq 2^{|X| \cdot 4^n}$, and take the Setting \ref{this:master:set} $(X,Y,g,h,k)$ for arbitrary $g \in F(X) \setminus \{e\}$ and $h \in F(X)$. \ey Then, any {\fns} $\nice{U} = \{U_i : i \leq n \}$ on $F(X)$ admits an extension to a {\fns} $\nice{V} = \{V_i : i \leq n\}$ on $F(Y)$ such that $g_0 \in V_n$. 
\end{lemma}

\begin{proof}
Let us set $\nice{V}' = \{V'_i: i \leq n\}$ and $\nice{V} = \{V_i: i \leq n\}$ be the respective $B' = \{e\}$ and $B = \grp{g_0}$ enrichments of $\nice{U}$ as in Lemma \ref{this:lemma:oct26}. \by By \eqref{eq:def:Vn} \ey we have $g_0 \in V_n$. Recall that $\nice{V'}$ is an extension of $\nice{U}$ by setting $C = \{e\}$ in Lemma \ref{cyclic:C-enrichment}. Since the inclusion $U_i \subseteq V_i$ always holds for every $i \leq n$, \by it suffices \ey to prove that $V_i \cap F(X) \subseteq U_i$ holds for every $i \leq n$. Let $\eta: F(Y) \to F(Y)$ be the map defined in Lemma \ref{reduction:lemma}. 

Let $h \in V_i \cap F(X)$ be arbitrary and assume $h = a_1 \fprod a_2 \fprod \cdots \fprod a_m$ is a canonical representation of $h$. 
Define the set \begin{equation}
L = \{l \in \{1,\dots,m\} : \eta(a_l) \neq a_l \}.
\end{equation}
\by The proof of the following claim is straightforward: 
\begin{claim} \label{this:claim:0ct312024}
$L = \{l \in \{1,\dots,m\} : a_l \in \grp{g_0} \setminus \{e\}\}$. %
\end{claim}
\ey

By Claim \ref{this:claim:0ct312024} we now satisfy all the conditions of Lemma \ref{workingman:lemma:7.2}, allowing us to deduce that \begin{equation} \label{this:eq:100:opct312024}
h = \eta(a_1) \fprod \eta(a_2) \fprod \cdots\fprod \eta(a_m).
\end{equation} 
Since $g_0$ contains a full instance of $f$ in its irreducible form, Lemma \ref{this:lemma:oct26} implies that $$(B \setminus B') \cap (\dash{X} \cup \bigcup_{i=1}^n V'_i) \subseteq \{e\}.$$
We now apply Lemma \ref{reduction:lemma} to deduce that $\eta(a_1) \fprod \eta(a_2) \fprod \cdots\fprod \eta(a_m)$ is the canonical representation of an element in $V'_i$. Since $\nice{V'}$ is an extension of $\nice{U}$, from \eqref{this:eq:100:opct312024} we deduce that \begin{equation*}
h = \eta(a_1) \fprod \eta(a_2) \fprod \cdots\fprod \eta(a_m) \in V'_i \cap F(X) = U_i.
\end{equation*}
This proves that $V_i \cap F(X) \subseteq U_i$ holds for arbitrary $i \leq n$, showing that $\nice{V}$ extends $\nice{U}$. 
\end{proof}

\section{Proof of Lemma \ref{topological:simplifity:set}}
\label{final:final:proofsection}

\begin{proof}
Let $n \in \N$, $S \in [X]^{< \omega}$, $g \in F(X) \setminus \{e\}$ and $h \in F(X)$ be arbitrary. Let us prove that \by the set \ey \begin{equation*} 
E_{n,S,g,h} = \{q \in \P: n \leq n^q, S \subseteq X^q \text{ and } (\mathrm{Conj}(g) \cdot h) \cap U_{n^q}^q \neq \emptyset \}.
\end{equation*}
is dense in $\P$. Take an arbitrary $p \in \P$. We employ already proven dense sets from Lemma \ref{dense:sets} to simplify matters. Select some $r \in \P$ for which $n \leq n^r$ and $S \subseteq X^r$ hold, $g,h \in F(X^r)$ and $r$ is an extension of $p$. Since $X$ is infinite, let us select $Y \subseteq X$ such that $| Y \setminus X^r| \geq 2^{|X^r| \cdot 4^n}$.  Let us enumerate $Y \setminus X^r = \{y_1, \dots, y_k\}$ and let $f = \prod_{i=1}^k y_i$. If we now consider \by $g_0 = f \cdot g \cdot f^{-1} \cdot h \in \mathrm{Conj}(g) \cdot h$, \ey Lemma \ref{this:corollary:oct312024:main} allows us to find an extension $\nice{V} = \{V_i : i \leq n^r\}$ for $\nice{U}^r$ on $F(Y)$ such that $g_0 \in V_{n^r}$. This inclusion implies that \by $$(\mathrm{Conj}(g) \cdot h) \cap V_{n^r} \neq \emptyset.$$ \ey
\by Let us now define $q = \langle\!\langle X^q,n^q,\nice{U}^q \rangle\!\rangle  \in \P$ as follows: \ey set $X^q = Y$, $n^q = n^r$ and $\nice{U}^q = \nice{V}$ where given $i \leq n^q = n^r$ one has $U_i^{q} = V_i$. From this re-indexing, we have $(\mathrm{Conj}(g) \cdot h) \cap U^q_{n^q} \neq \emptyset$. Since the conditions $n \leq n^q$ and $S \subseteq X^q$ also hold, this proves that $q \in E_{n,S,g,h}$. \by Since $q$ extends $r$, and $r$ itself extends $p$, we deduce that $q$ extends $p$, thereby proving the density of $E_{n,S,g,h}$. \ey  
\end{proof}

\end{document}